\documentclass[12pt]{amsart}
\usepackage[utf8]{inputenc}
\usepackage{amsfonts,amsmath,amssymb,amsxtra,stmaryrd}

\usepackage{a4wide}

\usepackage{colonequals}
\usepackage[utf8]{inputenc}
\usepackage[british]{babel}
\usepackage{hyperref}
\hypersetup{colorlinks=true,urlcolor=blue,citecolor=blue,linkcolor=blue}

\usepackage[all,cmtip]{xy}
\usepackage{float}

\usepackage[dvipsnames]{xcolor}
\usepackage[normalem]{ulem}
\usepackage{cleveref}
\Crefname{equation}{}{}

\newcommand{\ModuliSpace}{T}
\newcommand{\ModuliParameter}{t}
\newcommand{\LastCoord}{W}

\usepackage{tikz-cd}
\usetikzlibrary{patterns,shapes,cd,arrows}
\usepackage[auto]{contour}
\contourlength{2pt}

\newcommand{\C}{\mathbb{C}}
\newcommand{\Q}{\mathbb{Q}}
\newcommand{\A}{\mathbb{A}}
\renewcommand{\P}{\mathbb{P}}
\newcommand{\orderThreeAuto}{g}

\def\A{{{\mathbb A}}}

\def\Z{{{\mathbb Z}}}

\def\P{{{\mathbb P}}}
\def\G{{{\mathbb G}}}

\def\O{{{\mathcal O}}}

\newcommand{\Jac}{\operatorname{Jac}}

\usepackage{mathrsfs}
\usepackage{tikz-cd}
\newcommand{\pullbackcorner}[1][dr]{\save*!/#1-1.7pc/#1:(-1.5,1.5)@^{|-}\restore}

\numberwithin{equation}{section}

\newtheorem{theorem}[equation]{Theorem}

\newtheorem{lemma}[equation]{Lemma}

\newtheorem{problem}[equation]{Problem}

\newtheorem{proposition}[equation]{Proposition}
\newtheorem{question}[equation]{Question}

\theoremstyle{definition}
\newtheorem{definition}[equation]{Definition}
\newtheorem{setup}[equation]{Setup}

\theoremstyle{remark}
\newtheorem{remark}[equation]{Remark}

\title{Examples of effectivity for integral points on certain curves of genus 2}

 \author{Pietro Corvaja}
 \address{Dipartimento di Scienze Matematiche, Informatiche e Fisiche, Universit\`a degli Studi di Udine, via delle Scienze 206, 33100 Udine, Italy}
 \email{pietro.corvaja@uniud.it}

 \author{Davide Lombardo}
 \address{Dipartimento di Matematica, Universit\`a di Pisa, Largo Bruno Pontecorvo 5, 56127 Pisa, Italy}
 \email{davide.lombardo@unipi.it}

 \author{Umberto Zannier}
 \address{Scuola Normale Superiore, Piazza dei Cavalieri 7, 56126 Pisa, Italy}
 \email{umberto.zannier@sns.it}

\date{}

\begin{document}

\begin{abstract}
We consider families of smooth projective curves of genus 2 with a single point removed and study their integral points. We show that in many such families there is a dense set of fibres for which the integral points can be effectively determined. Our method is based on the construction of degree-3 étale covers of such curves of genus 2 and the study of the torsion values of sections of certain doubly elliptic abelian schemes.
\end{abstract}

\maketitle


\medskip


\section{Introduction} 

After the negative solution of Hilbert's tenth problem, obtained by Matijasevic based on the work of Davis, Putnam and Robinson, it was gradually realised that problems of effectivity in Diophantine equations, even in special cases (for instance concerning curves), could be very hard.

In the case of integral points on affine curves, Siegel's theorem leads to an effective algorithm {\it only} for deciding whether such a set is infinite or not: the theorem predicts finiteness unless the curve has genus 0 and at most two points at infinity. In these cases, an effective procedure has long been known that gives a complete parametrisation of the solutions in terms of polynomial or exponential functions. 
However, if we establish that such a set is finite, there is to date no known general algorithm to compute the finite set of solutions, or even to decide whether it is non-empty.\footnote{We mean that there are no known algorithms that can be proved to terminate, whereas there are algorithms which conjecturally terminate. See for example \cite{alpöge2024conditionalalgorithmicmordell} for a recent conjectural algorithm to determine the set of rational points on a smooth projective curve of genus at least 2.}

Such algorithms do exist for curves of genus $0$ or $1$, or for curves of certain other special shapes. Most of the methods rely on effective estimates for linear forms in logarithms, first developed by A.~Baker (later E.~Bombieri proposed alternative treatments, with methods of different nature).

One of the simplest instances for which the mentioned decidability problems are still open occurs for the quartic curves of genus $2$ of equation
\begin{equation}\label{E.quartic}
y^4+ay^2+xy+x^3+bx^2=0
\end{equation}
where $a,b$ are arbitrary algebraic integers:  the problem of finding (the) solutions in a given ring of $S$-integers is -- to our knowledge -- not completely solved.

\begin{remark}\label{rmk: singular points}
    Every curve of genus $2$ is birational to a curve of the form $y^2=f(x)$, where $f(x)$ is a sixth-degree polynomial, and if this were the given affine model then all the integral solutions $(x,y)$ could be computed by the methods mentioned above. In this hyperelliptic model, however, the underlying affine curve has two points at infinity, which moreover are interchanged by the canonical involution. In contrast, the quartic model \eqref{E.quartic} has a single point at infinity, which, for general values of $a,b$, is not a special point. (Note that the quartic curve is singular at the origin, which does not affect the problem of integral points; namely, apart from a finite number of easily determined exceptions corresponding to the singular point, the integral points of a smooth model correspond bijectively to the integral points of the present model.)
\end{remark}

To our knowledge, there are no algorithms in the literature for solving such equations over arbitrary number fields, even for special values of the parameters (unless $a=b=0$).

\medskip

\subsection{Background and goals}

The present article concerns a method to obtain effectivity for the search of integral points on certain (sets of) curves of genus $2$, with a single point at infinity; in particular, this will include infinitely many cases in the family \eqref{E.quartic}.  We shall develop certain geometrical constructions aimed at applying a criterion for effectivity due to Y.~Bilu, which will be recalled soon.

\medskip

{\bf Some terminology and background}. 
For a ring $\O_S=\O_{K,S}$ of $S$-integers in a number field $K$ and a curve $X$ over $K$ embedded in some affine space, we shall denote by $X(\O_S)$ the set of points in $X$ with coordinates in $\O_S$. 
This definition depends on the embedding, but since we will vary $K$ and $S$ freely, the embedding is actually irrelevant for our purposes.
We say that {\it $X(\O_S)$ is effective} if it is finite and computable\footnote{In each case when this set is infinite, this may be established and a finite parametrisation can be found.}.

For any of the curves $X$ covered in our treatment, $X(\O_S)$ will be effective for any ring $\O_S$ as above and any embedding of the curve.

\begin{remark}\label{rmk: descent does not matter}
    Let $X$ be an affine curve defined over $\overline{\Q}$. There will be many number fields $K$ for which there exists a model $X_K$ of $X$ over $\mathcal{O}_{K, S}$. Moreover, even if we fix $K$, there will be multiple models of $X$ defined over $\mathcal{O}_{K,S}$. In our setting, we will have effectivity for all models of our curves, defined over any ring of the form $\mathcal{O}_{K, S}$. Note that, if we are given two models $X_1 / \mathcal{O}_{K_1, S_1}$ and $X_2 / \mathcal{O}_{K_2, S_2}$, there exists a number field $K_3$ containing $K_1, K_2$ and a finite set of places $S_3$ of $\mathcal{O}_{K_3}$ (containing all those that divide a place in $S_1$ or $S_2$) such that $X_1$ and $X_2$ become isomorphic to each other over $\mathcal{O}_{K_3, S_3}$. Since the $\mathcal{O}_{K_1, S_1}$-points of $X_1$ embed in its $\mathcal{O}_{K_3, S_3}$-points, and similarly for $(K_2, S_2)$, effectivity for $(K_3, S_3)$ gives effectivity for $(K_1, S_1)$ and $(K_2, S_2)$. For this reason, in the context of our problem, we can fix any (number) field of definition $K$ for $X$ and any model. Of course, when we say that $X(\mathcal{O}_{K,S})$ is effective for all number fields $K$, we implicitly consider only those number fields $K$ that contain at least \textit{some} field of definition for $X$.
    
    The upshot of this discussion is that the property ``$X(\mathcal{O}_{K,S})$ is effective for all $K$ and $S$" (in the sense above) is really a property of the $\overline{\Q}$-isomorphism class of $X$. We will use this fact repeatedly and without further mention when we talk about moduli spaces of curves over $\overline{\Q}$. See also \Cref{sect: models of genus 2 curves} for a more extended discussion.
\end{remark}

\smallskip

Before discussing in any detail our approach, let us pause to recall a few other known effective methods for curves. 

\smallskip

{\it Reduction steps}. First, note that if we have a non-constant morphism $\varphi:Y\to X$ defined over $\O_S$ between two affine irreducible curves, and if $X(\O_S)$ is effective, then $Y(\O_S)$ is effective as well; this reduction step, though obvious, is often useful.  Another, more subtle, reduction occurs via the Chevalley-Weil theorem concerning unramified covers $Y\to X$ of affine curves. In this case, we go in the opposite direction, by lifting integral points to $Y$, so we need effective knowledge for $Y$ rather than $X$.  In this way, the implication is a little weaker: if $Y(\O_S)$ is effective for all $K,S$, then the same holds for $X$.

\medskip

Using these principles, the known cases of effectivity are often reduced to the case of $\P_1\setminus\{0,1,\infty\}$, which may be written explicitly as the equation $x+y=1$, to be solved in $S$-units, that is, with $x,y\in\O_S^*$.

This equation was first solved effectively by Baker (and then by Bombieri using a different method); in fact, these methods in particular lead to the effective solution of the inequality
\begin{equation}\label{eq: Baker inequality}
    |\alpha-g|_\nu <H(g)^{-\kappa}
\end{equation}
for any fixed non-zero algebraic number $\alpha$, to be solved for $g\in\Gamma$, where $\Gamma\subset \overline\Q^*$ is a finitely generated multiplicative group, $\nu$ is a valuation of the number field  $\Q(\alpha,\Gamma)$ and $\kappa>0$. \footnote{In fact, the bounds for the height of the solutions are completely explicit in terms of basic functions.} This inequality is fundamental to derive a criterion of Bilu, which will be recalled shortly, and which, as we have said, is the basis of our method.

\begin{remark}
The $S$-unit equation has been effectively solved by completely different remarkable methods. 
On the one hand, M.~Kim introduced in 2005 \cite{MR2181717} a sophisticated evolution of the $p$-adic method of Skolem-Chabauty (which is still widely studied); on the other hand, a few years later, K.~Murty and H.~Pasten \cite{MR3084298}, and independently R.~von K\"anel \cite{MR3296485}, used a modular approach to re-obtain effective estimates, at least over $\Q$. von K\"anel and collaborators have also extended this approach to obtain other effectiveness results, see \cite{MR4587796, MR4264210, VONKANEL2024}. A modular approach to effectiveness has also been put forward by L.~Alp\"oge \cite{alpöge2021modularityeffectivemordelli}.

We note that these methods, powerful as they are, do not seem to imply an inequality of the strength of \eqref{eq: Baker inequality}, which is however needed for Bilu's criterion.
Other results on effectivity, which often boil down to $S$-unit equations, are due to Grant \cite{MR1184116}, Levin \cite{MR2477511, MR3782466}, and Levesque-Waldschmidt \cite{MR3284122}.
\end{remark}

\medskip

{\it Bilu's criterion}. 
This is based on the following result of independent interest: {\it For every irreducible curve $Z\subset\G_m^2$, not a translate of a subtorus, and for every $K,S$, the set $Z(\O_S)$ is effective.}

In other words, if $f(x,y)$ is an irreducible polynomial not defining a translate of a subtorus, which amounts to $f(x,y)$ having at least three monomials, then the equation $f(x,y)=0$ can be effectively solved in $S$-units $x,y\in\O_S^*$. 
A proof \cite{Bi} can be derived by suitably applying inequality \eqref{eq: Baker inequality} to a Puiseux expansion of the coordinate functions on $Z$. We also refer to the book by Bombieri-Gubler \cite{BG} for an even more direct argument. Note that the $S$-unit equation corresponds to taking $Z$ to be defined in $\G_m^2$ by the equation $x+y=1$. Combining this result with the above reduction steps we obtain:

\smallskip

{\bf Bilu's criterion}: {\it Let $X$ be an affine curve over $\bar{\Q}$ such that there exist an \'etale cover $Y\to X$ and a non-constant morphism $Y\to \G_m^2$ whose image is not a translate of a subtorus. Then $X(\O_S)$ is effective for every $K,S$.}
\medskip

As announced above, we will apply this criterion to certain affine curves $X/\overline\Q$ of genus $2$ with smooth projective model $\tilde X$, where $X=\tilde X-\{q_0\}$ and $q_0$ is some (non-special, i.e. not fixed by the canonical involution) point on $\tilde X$.

The results we obtain by no means give effectivity for all such curves $X$. However, they apply to a set of curves (and points $q_0$) that is, for example, dense in the moduli space of such data for the complex-analytic topology. See Theorems \ref{T} and \ref{thm: two-parameters family intro} below for a flavour of the statements we can obtain, and the comments on the method at the end of \S\ref{SS.hints} and in \S\ref{sect: comments}. Our main purpose is to show that Bilu's criterion is sometimes successful, despite work of Landesman and Poonen \cite{LP} indicating strong limitations.

To apply the criterion, we shall construct suitable \'etale covers $Y\to X$ and morphisms $Y\to \G_{\rm m}^2$, with images of increasing degree; this will only work for a subset of the parameters defining the curves. We note that, as this degree increases, we might say that the examples become `more interesting' since they cannot be derived by substitution from a universal continuous family. 
 On the negative side, the relevant curves will have increasing fields of definition.
 \medskip

 On the other hand, to our knowledge, no result of this kind has appeared in the literature before. In fact, the known cases of effectivity for curves of genus $2$, not of the special hyperelliptic form, seem to be sporadic; for instance, Serre's book \cite{Se} explicitly inquires about such possible effectivity  (see the remarks at the end of p.~116). The shape of the curves that we treat is somewhat significant (an example is represented by the family of equations \eqref{E.quartic}), being in a sense the simplest for which effectivity is not known in general.
Indeed, these curves can be realised as affine quartic curves of genus 2 with one point at infinity (see Lemma \ref{lemma: singular model with one point at infinity}), while effectivity is known for all affine curves of genus $\le 1$, in particular for all cubics.

\medskip

Furthermore, we do not know the exact limits of applicability of Bilu's criterion. In this direction, recent work by Landesman-Poonen \cite{LP} does indeed suggest strong limitations. However, we note that the results in \cite{LP} refer to curves over $\C$. In principle, by some kind of specialisation argument, one is led to expect that similar results should hold over $\overline\Q$; however, some caution is needed in trusting such a conclusion, as shown for instance by notable examples such as Belyi's theorem. We also mention the conjecture of Bogomolov-Tschinkel \cite[Conjecture 1.1]{MR2159376} that \textit{every} (connected) curve over $\overline{\Q}$ is dominated by an étale cover of $y^6=x^2+1$.

\medskip 
 
{\bf Our results}. Here and in the sequel we take $\overline\Q$ as our ground field.
Before stating our results, we recall some classical facts about curves of genus $2$. Any smooth projective (connected) curve $\tilde X$ of genus $2$ can be obtained by adding two points at infinity to an affine curve defined by an equation $y^2=f(x)$, where $f$ is a polynomial of degree $6$ without multiple roots. This exhibits $\tilde X$ as a double cover of $\P_1$ ramified at $6$ points, called {\it special points} or {\it Weierstrass points}. 
The datum of these points (or, more precisely, of their images in $\mathbb{P}_1$) up to the action of $\operatorname{PGL}_2$ determines $\tilde X$ up to isomorphism.\footnote{Sending one of these points to infinity leads to a birational model for $\tilde X$ of the same shape as above, but with $\deg f=5$.} There are fairly complicated {\it Igusa invariants} for the resulting quotient space. 

There is a natural (unique) involution on $\tilde X$ with quotient $\P_1$, defined by $(x,y)\mapsto (x,-y)$ and called the hyperelliptic involution, which fixes the six special points. 
We shall be interested in effectivity of integral points on subsets of the forms  $\tilde{X}-\{q_0\}$. 
 If $q_0$ is a special point then $X$ has a model in $\A^2$ given by an equation $y^2=f^*(u)$ where $f^*$ has degree $5$: in this case, the integral points can be found effectively by well-known methods\footnote{The criterion of Bilu that we use can be read as including these cases, but with a construction of a cover that goes back to Siegel and is different from the one we will introduce.}. The same effectivity holds for an affine curve defined in $\A^2$ by $y^2=f(x)$ with $\deg f=6$ (actually, any polynomial $f$): however, in this last affine model we remove {\it two} points from $\tilde X$, and these two points are moreover exchanged by the hyperelliptic involution.
Instead, our examples will all involve a single non-special point $q_0$. 

\medskip

Smooth projective algebraic curves of genus $2$ are parametrised by a quasi-projective variety, usually denoted $M_2$ (recall that we work over $\overline{\mathbb{Q}}$). This is a rational variety of dimension $3$. In the sequel, the terminology `moduli space for curves of genus $2$' will refer to some quasi-projective variety $M$ with a birational map to $M_2$. Thus, generically, a point of $M$ will represent a curve of genus $2$, and every curve of genus $2$, with the exception of a family at most of dimension $2$, will correspond to a point of $M$. We will also call $M_{2,1}$ the quasi-projective variety parametrising pairs $(X,p)$, where $X$ is a curve of genus $2$ and $p$ is a point in $X$; it has dimension $4$.

Our results concern families of equations (or, more abstractly, families of curves). While the statement we give below is slightly more general, the reader may find it useful to think about the following situation. Consider a family $\tilde{X}_\ModuliParameter$ of (possibly singular) projective curves of geometric\footnote{One way to think about this condition is that the unique smooth projective curve birational to $\tilde{X}_\ModuliParameter$ has genus 2.} genus 2, given by equations whose coefficients depend algebraically on finitely many parameters $\ModuliParameter=(\ModuliParameter_1, \ldots, \ModuliParameter_n)$. Suppose furthermore that -- in the given projective embedding -- these curves have a single point at infinity (depending on $t$), that we denote by $q_t$. For example, the projective closure of the family \eqref{E.quartic} has these properties, with $\ModuliParameter=(a,b)$ and $q_t=[1:0:0]$. We may then ask about the integral points of the affine curve $\tilde{X}_t \setminus \{q_\ModuliParameter\}$. This situation can be formalised as follows: the family $\tilde{X}$ is a subvariety of $T \times \mathbb{P}_m$, where $\ModuliSpace$ (the space of parameters $t$) is some affine variety. The fibres of the natural projection $\tilde{X} \to \ModuliSpace$ are projective curves of genus 2 in $\mathbb{P}_m$, and the marked point $q_\ModuliParameter$ amounts to an algebraic section $q: \ModuliSpace \to \tilde{X}$ of this projection. Finally, note that for every $t \in T$ we can consider the smooth projective curve birational to $\tilde{X}_t$: this corresponds to a point in the moduli space $M_2$, so that we obtain a map $\ModuliSpace \to M_2$.

An example of our results is the following theorem:

\begin{theorem}\label{T}
Consider a variety $\ModuliSpace$ equipped with an algebraic family $\tilde{X} \to \ModuliSpace$ of projective curves of geometric genus 2 and a generically non-special section $q : \ModuliSpace \to \tilde{X}$. Suppose that the induced moduli map $\ModuliSpace \to M_2$ (sending each point $\ModuliParameter \in \ModuliSpace$ to the point in $M_2$ corresponding to the smooth curve birational to $\tilde{X}_\ModuliParameter$) is dominant and of finite degree. Suppose furthermore that $q_{\ModuliParameter}$ is generically a smooth point of $\tilde{X}_{\ModuliParameter}$.
There is a complex-analytically dense subset $\Sigma\subset \ModuliSpace$ of algebraic points on $\ModuliSpace$ such that for each $\ModuliParameter\in\Sigma$ the integral points on $X_\ModuliParameter=\tilde{X}_\ModuliParameter-\{q_\ModuliParameter\}$ are effectively computable, over every number field.
\end{theorem}

\begin{remark}\label{rmk: remark to main theorem}
\phantom{}
\begin{enumerate}
\item Since we are working with the coarse moduli space $M_2$ over $\overline{\Q}$, the curve $\tilde{X}_t$ is a priori only defined over $\overline{\Q}$. As explained in \Cref{rmk: descent does not matter}, this does not affect the conclusion, since we can fix an arbitrary descent to a number field.
\item The set $\Sigma$ is itself `effective', in the sense that we may compute some element of $\Sigma$ in any prescribed open disk in $\ModuliSpace$, and we may establish whether any given algebraic point of $\ModuliSpace$ belongs to $\Sigma$. The set $\Sigma$ is described by countably many explicit algebraic equations. The nature of these equations will be clearer from the arguments and the examples below. The points of $\Sigma$ correspond to points of $\ModuliSpace$ where a certain section of an abelian scheme on (a finite cover $\ModuliSpace'$ of) $\ModuliSpace$ takes values which are torsion in the corresponding fibre. 
\item We can see $\ModuliSpace$ as a finite-degree cover of a (locally closed) subvariety of $M_{2,1}$ whose projection to $M_2$ is of finite degree and has open dense image.
\item We note explicitly that the set $\Sigma$ is not $p$-adically dense, for any $p$. Indeed, as already indicated, $\Sigma$ is the locus where a certain section of an abelian scheme becomes torsion, and it is known \cite{MR4120270} that the torsion specialisations of a non-torsion section are not $p$-adically dense.
\item For every point $t\in\Sigma$ we have an example of effectivity for $X_t(\O_S)$ explained by Bilu's criterion, so there is nothing new in any of these individual examples. Our objective here is to exhibit many such examples where the applicability of Bilu's criterion becomes increasingly more difficult to predict, which also shows its potential. We note that our construction is, in a sense, the simplest possible. It probably allows for variations that could produce further examples, perhaps with small fields of definition.
\end{enumerate}
\end{remark}

\begin{question}\label{question: non-split Jacobian}
        It would be interesting to show that the set of curves for which the methods of the present paper give effectivity is not contained in the set of curves with split Jacobian (we fully expect this to be the case, but it is not clear to us what is the simplest way to show this statement). Note, however, that even if $\tilde{X}$ admits a map $\pi : \tilde{X} \to E$ to an elliptic curve (i.e., the Jacobian of $\tilde{X}$ is split), it is not a given that one can determine the integral points of $\tilde{X} \setminus \{x\}$, because they map to integral points of $E \setminus \pi(x)$ only if $\pi^{-1}(\pi(x))=x$, that is, if $\pi$ is totally ramified at $x$. By the Hurwitz genus formula, this can only happen for $\deg \pi\leq 3$. This shows that the curves $\tilde{X}$ for which there is effectivity for the integral points due to the existence of a map to an elliptic curve are contained in a proper subvariety (the locus of curves with $(2,2)$ or $(3,3)$-split Jacobian). In contrast, our method gives effectivity for a set of curves that is dense in the moduli space.
\end{question}

\Cref{T} has the assumption that the moduli map $\ModuliSpace \to M_2$ is dominant. More generally, we will be able to show that the conclusion of \Cref{T} holds for many $2$-parameter families of curves. We formalise this as \Cref{thm: dichotomy for families}, which, combined with a short calculation, gives for example:
\begin{theorem}\label{thm: two-parameters family intro}
    Let $q_{a,b}$ be the unique point at infinity in the projective completion $\tilde{X}_{a,b}$ of 
    \[
    X_{a,b} : y^4+ay^2+xy+x^3+bx^2=0.
    \]
    There is a complex-analytically dense set of algebraic points $(a,b) \in \overline{\Q}^2$ such that the integral points on $X_{a,b} \colonequals \tilde{X}_{a,b} \setminus \{q_{a,b}\}$ can be determined effectively, over every number field.
\end{theorem}

To conclude this introduction, we recall some related work. As already mentioned, our method is based on Bilu's criterion (see \Cref{thm: Bilu version 1} below), which may be expressed in terms of (an étale cover of) the pointed curve $\tilde{X} \setminus \{x\}$ having two independent morphisms to $\mathbb{G}_m$. Example 1.6 of \cite{LP} (see also the discussion before Proposition 6.1 of \cite{Bi}), due to Ikehara and Tamagawa, shows that for every curve $\tilde{X}$ of genus $2$ over $\overline{\Q}$ there exist infinitely many points $x \in \tilde{X}$ such that $\tilde{X} \setminus \{x\}$ has a finite étale cover with a non-constant morphism to $\mathbb{G}_m$. In a different direction, work of Grant \cite{MR1184116} relates the integral points on curves of genus 2 to the integral points on certain affine subsets of their Jacobian, and in turn, reduces the latter problem to the study of what he calls non-abelian $S$-unit equations. Yet other approaches to effectivity rely on versions of Runge's theorem (see for example \cite{MR2477511, MR3782466} for sophisticated modern versions of this approach) or on the relation to the Shafarevich conjecture \cite{MR3010636}.

\begin{remark}
Some of the results in this paper have been obtained with the help of the computer algebra system MAGMA. All the code to verify our computational claims can be found online, see \cite{Computations}.
\end{remark}

\subsection{Some explicit examples}

 A curve of genus $2$ with a marked point $q_0$ may be represented as a quartic in $\A^2$, with a single smooth point at infinity (corresponding to $q_0$) and precisely one singular point, see Lemma \ref{lemma: singular model with one point at infinity}. The shape of these equations recasts the problem we consider in classical diophantine terms.

Let us see some examples of such quartics. Consider the family of curves of genus $2$ given by the (singular) equations
\begin{equation}\label{eq: two-parameter example}
X_{a,b} : y^4+ay^2+xy+x^3+bx^2=0,
\end{equation}
with parameters $a,b$. Generically, the curve $X_{a, b}$ is singular (only) at the origin and has a single smooth point at infinity, in a closure in $\P_2$. The genus of $X_{a, b}$ is $2$ by the formula $g=(d-1)(d-2)/2-s$, where $d$ is the degree of the equation and $s$ is the number of singular (nodal) points.
This is the family we will study in the proof of \Cref{thm: two-parameters family intro}.
For the moment, we limit ourselves to pointing out a very special family of singular plane quartics on which the integral points can be found effectively, and we mention a related problem in diophantine effectivity.

\begin{remark}[Hyperelliptic model] Intersecting the quartic \eqref{eq: two-parameter example} with the pencil of lines through the singular point (the origin), one finds the hyperelliptic genus $2$ involution, deriving from switching the other two intersections of the quartic with such a line:
\begin{equation*}
(x,y)\mapsto \left( 2{x^4\over y^4}-x, 2{x^3\over y^3}-y \right).
\end{equation*}
In this way we obtain the usual sextic model
\begin{equation}\label{eq: sextic model}
\mu^2=a\lambda^6+\lambda^5+b\lambda^4-1,\qquad \lambda={y\over x},\quad \mu=\lambda^{-4}-x.
\end{equation}
Of course, this is to be interpreted as defining the function field of the relevant curves. The affine model given by equation \eqref{eq: sextic model} has two points at infinity.
\end{remark}

\subsubsection{An effective case: trinomial equations} In a certain expository paper by Masser-Zannier, seeking maximum simplicity, the example $y^4-axy-x^3=0$ was proposed. However, the authors did not notice that the existence of the automorphism of order $5$ given by ($x\to\theta x, y\to \theta^2y$), where $\theta$ is a fifth root of unity, leads to the affine model $y^5=w(w+a)^2$, where $w=x^2/y$. The integral points on this curve can of course be determined by the usual methods. 
(Bogdan Grechuk pointed out this fact to Masser and Zannier, with a different and much longer series of substitutions.) 
This peculiar coincidence does not occur in the examples below, nor for the general equation \eqref{eq: two-parameter example}, as it is not difficult to check.
\medskip

Trinomial equations are actually \textit{always} effectively solvable. Over the usual ring of integers $\Z$, this was observed and proved by B.~Grechuk, T.~Grechuk, and A.~Wilcox \cite{GGW}. Their proof is based on Runge's theorem and on Baker's results on superelliptic equations, and requires a rather long argument with comparison of factorisations and distinctions in cases. In particular, due to the use of Runge's theorem, their proof, as it stands, does not work over general number fields, but only over $\Z$ (the rest of their argument also does not take into account the possible existence of non-trivial units); however, it would be possible to modify their proof for the general case.  

In essence, the argument of \cite{GGW} exploits the Newton polygon to express $x,y$ multiplicatively in terms of two other functions. A shorter, different, and completely general proof is as follows. We emphasise again that the case $\mathcal{O}_{K, S}=\mathbb{Z}$ of this statement was originally proved by Grechuk, Grechuk, and Wilcox in \cite{GGW}.
\begin{theorem}[Generalisation of \cite{GGW} with a different argument]
    Let $K$ be a number field and $S$ be a finite set of places of $K$. Consider the equation 
    \begin{equation}\label{eq: general trinomial}
            x^n+ax^ry^s+by^m=0,
    \end{equation}
    where the exponents $m, n, r, s$ are non-negative integers and the coefficients $a,b$ are non-zero elements of $K$. One can effectively decide whether \eqref{eq: general trinomial} has infinitely many $S$-integral solutions; when the set of solutions is finite, it can be effectively determined.
\end{theorem}

\begin{proof}
The equation $x^n+ax^ry^s+by^m=0$ defines an affine curve $Z$. Define functions $u,v$ on $Z$ by $u:=x^n/y^m$ and $v:=x^r/y^{m-s}$; note that $u+av+b=0$. Set $\Delta=|nm-rm-sn|$; this is the absolute value of the determinant of the matrix $\begin{pmatrix}
    n & r \\ m & m-s
\end{pmatrix}$.   If $\Delta=0$ then $u,v$ are multiplicatively dependent, and the equation $u+av+b=0$ shows that the curve $Z$ is a union of translates of tori in $\G_m^2$, so the problem is standard. If $\Delta\neq 0$, the curve $Z$ is irreducible, as can be seen from the substitution $x\mapsto t^m x, y\mapsto t^ny$, which transforms the equation into a binomial in $t$.\footnote{See Schinzel's book \cite{S} for a complete and subtle theory of reducibility of trinomials over function fields, not needed in the simple case considered here.}

There is an action of an abelian group $G$ on the function field $\overline{K}(Z) = \overline{\Q}(Z)$ of $Z$, defined as follows. Let
\[
G = \left\{ (\theta, \eta) \in \overline{\Q}^\times : \theta^\Delta=\eta^\Delta=1, \; \theta^n=\theta^r\eta^s=\eta^m \right\}.
\]
The group $G$ acts on $\overline{K}(Z)$ by $(x,y)\mapsto (\theta x,\eta y)$. 
Writing $\theta=\zeta^p,\eta=\zeta^q$ for a fixed primitive $\Delta$-th root of unity $\zeta$, the equations $\theta^n=\theta^r\eta^s=\eta^m$ correspond to $pn-qm\equiv pr-q(m-s)\equiv 0\pmod\Delta$. Hence, the group $G$ corresponds to integer vectors $(p,q)$ satisfying these congruences, which form a subgroup of $\Z^2$ of index $\Delta$, modulo $\Delta \Z^2$. The action of $G$ leaves the functions $u,v$ invariant. In fact, the field of invariants is generated by $u$ and $v$, because the lattice generated by $(n,m)$ and $(r,m-s)$ has index $\Delta$ in $\Z^2$.  Since $u,v$ are linearly related,  the field of invariants corresponds to a curve of genus $0$. Also, the group action induces an action on the points of a smooth complete model $Z'$ of $Z$, which leaves invariant the set of poles of $x$. Hence, the affine curve for which we seek the integral points may be seen as a fibre product of two cyclic covers of the affine line (possibly with some points removed), and thus the known methods allow one to determine the $S$-integral points, if their set is finite, or to parametrise them. It turns out that there may be infinitely many integral points only if $G$ is cyclic, which happens if and only if $\gcd(m,n,r,s)=1$.
\end{proof}

\begin{remark}
The case of general Galois covers of the affine line is less direct than the special case appearing here. It can still be treated effectively, as in the independent results of Bilu \cite{Bi} and Dvornicich-Zannier \cite{DZ}. In fact, both of these results appeared in earlier publications that are very difficult to find today. Another independent argument was given by Poulakis. Serre interpreted all of these arguments in a short proof, in a letter to  Bertrand on 4 June 1993.
  See also \cite{Z2} for a complete and self-contained description of the general case of genus $0$.
\end{remark}

\begin{remark} Similar argument can also be applied in other cases, for instance when $u,v$ define a conic, or a curve of genus $\le 1$, for which the points at infinity are the $G$-orbits of the poles of $x$. 
\end{remark}

\subsubsection{Concrete examples}\label{subsubsec: concrete example}
Consider the curve $X$ with planar model
\begin{equation}\label{eq: first explicit example}
324x^4 - 324x^3 - 18x^2y + 71x^2 - 10xy - 6y^3 - 12y^2 = 0.
\end{equation}

This plane quartic has a unique singular point at $(0, 0)$ and a unique point $q$ at infinity (see also Lemma \ref{lemma: singular model with one point at infinity} for a discussion of the general situation). Therefore, it represents a curve of genus 2 with a marked point (at infinity). We will explain in \S\ref{subsec: concrete example, details} how our method gives effectivity for $X(\O_S)$.

\begin{remark}
    The way we build such examples, which we will discuss in greater detail in \S\ref{sect: examples}, starts by considering a family of genus-2 curves $X_\alpha$ over a base $\ModuliSpace$. This also gives the family of  Jacobians $J_{\alpha}$ of $X_\alpha$ (seen as an abelian scheme over $\ModuliSpace$), and we shall construct a suitable section $\sigma_{\alpha}$ of $J_\alpha$ for $\alpha\in\ModuliSpace$. The desired examples are found by considering those values of $\alpha$ for which $\sigma_\alpha$ is a torsion point of $J_\alpha$. 
    By classical results of Silverman-Tate, and more recent ones by Dimitrov-Gao-Habegger and Yuan-Zhang, the ``generic'' suitable values of $\alpha$ have bounded height, so their degree and torsion orders must tend to infinity outside certain proper closed subvarieties. We also recall that results by Mazur and Merel  \cite{Mazur, Merel} restrict the torsion order in terms of the degree of the field of definition.
     These facts explain why it is difficult to find examples having a small field of definition.
\end{remark}

For the curve $X$ of \eqref{eq: first explicit example}, the two components of the relevant section of the doubly-elliptic abelian scheme are both torsion of order 3. That they have the same order is explained by the fact that they are conjugate under the action of $\operatorname{Gal}(\overline{\Q}/\Q)$. A different situation arises with the curve given by (see \Cref{subsect: different torsion orders})
\[
81x^4 - 162x^3 + 9x^2y + 107x^2 - 62xy - 9y^3 + 36y^2=0.
\]
Here, the two component of the torsion section have order $2$ and $3$, and we again have effectivity for the integral points, over every number field.

\subsection{Perfect-power values of algebraic functions}

We point out a connection between the problem of effective integral points on curves of genus $2$ and another natural problem in diophantine theory.

\begin{problem}\label{prob: perfect power values}
Let $\pi: Y\to \A^1$ be a morphism from an affine curve $Y$ over a number field $K$, and suppose we are given a rational function $f\in K(Y)$. The problem is to find effectively the integral (or rational) points $m$ on $\A^1$  such that they lift to an algebraic point $\xi=\xi_m\in Y(\overline{K})$  with the property that $f(\xi)$ is a perfect $p$-th power in $K(\xi)$. The integer $p\ge 2$ may be fixed or variable, and the function $f$ is supposed not to be a perfect power in $K(Y)$.
\end{problem}

\begin{remark}
Let $Y=\A^1$ and $\pi$ be the identity. Now $f$ is simply a function in $K(x)$ and we have $\xi=m\in K$,  so we are looking for the points in $K$ where the value of an ordinary given rational function is a $p$-th power. Restricting to the integers of $K$, such values can be found effectively (using Baker's method, as applied by Schinzel-Tijdeman \cite{ST}), even if we let $p$ vary.
\end{remark}

To study Problem \ref{prob: perfect power values} in general, we can take the norm of $f$ down to $K(x)=K(\mathbb{A}^1)$ and obtain a rational function $F=F(x)$ on the affine line. It is then often possible to apply the techniques of Schinzel-Tijdeman to $F(x)$, see for instance the paper by B\'{e}rczes, Evertse and Gy\H{o}ry \cite{BerGyo}. However, a problem arises if $F$ is identically a perfect power -- for instance, a constant -- which can happen even if $f$ is not itself a perfect power.

For any fixed $p$, the problem of the finiteness of the integral points can be solved using Siegel's theorem, as follows.
Let $L$ be a normal closure of $K(Y)(f^{1/p})/K(x)$; we can consider $L$ as the function field of a smooth projective curve $Z$. Suppose that $f$ is not identically a $p$-th power of a rational function on $Y$. Up to enlarging $K$, we can further suppose that $K$ contains a primitive $p$-th root of unity. By definition, 
a quotient of $Z$ by a suitable decomposition group has an integral point lying over $\xi_m$, so we can decide effectively whether there are infinitely many integral points by applying Siegel's theorem to suitable quotients of $Z$.
It seems that for $p>2$ we always have finiteness. On the other hand, for $p=2$ there are certainly cases with infinitely many integral points, which can be found through Pell-type equations.

Motivated by the positive results for powers, especially those by Schinzel-Tijdeman, we ask:   

\begin{question}\label{question:perfect powers}
Is Problem \ref{prob: perfect power values} effectively solvable? Can we find these integral points when their set is finite?    
\end{question}

\subsubsection{Cubic roots} There is a connection between the case $p=3$ of Problem \ref{prob: perfect power values} and Question \ref{question:perfect powers} and the determination of integral points on curves of genus $2$. For instance, consider \eqref{eq: two-parameter example} as a cubic in $x$ and solve using the Cardano-Tartaglia formulae. Taking for instance $b=0$ in \eqref{eq: two-parameter example}, we have the equation 
\begin{equation*}
x^3+xy+y^4+ay^2=0.
\end{equation*}
The cubic formula expresses $x$ linearly in terms of the cube roots
\begin{equation*}
\root 3\of {-y^4-ay^2\pm y\sqrt{(y^3+ay)^2+(4/27)y}}.
\end{equation*}

Corresponding to an integral point $x=r,y=s$ on the original curve, the expression under the cube root gives a value in the quadratic field $\Q(\sqrt{(s^3+as)^2+(4/27)s})$.
The difficulty is that we cannot assume that the square root $\sqrt{(s^3+as)^2+(4/27)s}$ is rational; actually, this holds only for finitely many integral values of $s$.
The problem is to determine when the cubic root lies in this quadratic field. 

 In the above notation, $y$ is the coordinate on $\A^1$ and $Y$ is the curve $z^2=(y^3+ay)^2+(4/27)y$. The function $f$ on $Y$ is $-y^4-ay^2 + yz$.
Note that the norm of $f$ down to $K(y)$ is $-\frac{4}{27}y^3$, a perfect cube (at least if $4^{1/3} \in K$), so we fall into the exceptional cases mentioned above, for which the known theory does not provide an effective answer. These formulae provide an explicit link between this problem and the problem of effectivity for integral points on curves of genus $2$.

\medskip

\noindent\textbf{Acknowledgements.}
    We are grateful to Aaron Landesman and Bjorn Poonen for their interest in this work and useful comments on earlier versions.

\section{Curves of genus 2 and their degree-3 cyclic covers}\label{sec: details of the construction}
In this section we review the structure of 3-torsion points on the Jacobian of a genus-2 curve $X$ and explicitly describe the étale covers of $X$ that arise from the choice of such a $3$-torsion point.
Let $K$ be a field of characteristic $0$. (Most of what we do can be adapted to fields of characteristic greater than 5).
Let $X$ be a smooth projective curve of genus 2, explicitly represented by the hyperelliptic model
\begin{equation}\label{eq: affine patch}
y^2 = f(x) = a_6x^6 + a_5x^5 + \cdots +a_1 x + a_0,
\end{equation}
where $f(x) \in K[x]$ is a separable polynomial of degree $5$ or $6$.
Equation \eqref{eq: affine patch} only gives a model for an affine patch of $X$: we now describe the projective completion, which will be useful in what follows.
Writing $\tilde{f}(u) \colonequals u^6f(1/u)$ for the reciprocal polynomial of $f$, the open subscheme $\{u \neq 0\}$ of the affine curve
\begin{equation}\label{eq: second affine patch}
v^2 = \tilde{f}(u) 
\end{equation}
can be glued to the open subscheme $\{x \neq 0\}$ of the affine curve \eqref{eq: affine patch} along the isomorphism $u=1/x, v=y/x^3$. The points $u=0, v=\pm \sqrt{\tilde{f}(0)} = \pm \sqrt{a_6}$ on \eqref{eq: second affine patch} give rise to points $\infty_\pm$ on $X$, which we call the \textit{points at infinity} of $X$ (when $a_6=0$ there is a unique such point at infinity; in this case, we set $\infty_+=\infty_-$ to be this point). 
The projective curve with the same function field as the affine curve \eqref{eq: affine patch} is obtained by adding the points at infinity to the affine patch described by \eqref{eq: affine patch}.

We now discuss 3-torsion points on the Jacobian $J_X$ of $X$. A straightforward application of the Riemann-Roch theorem shows that any degree-0 divisor on  $X$ is linearly equivalent to one of the form $(P_1)+(P_2) - (\infty_+)-(\infty_-)$.
The following lemma describes which of these divisors represent $3$-torsion points. For a proof, see for example \cite[Lemma 3]{MR3263947}.
\begin{lemma}\label{lemma: 3-torsion points}
The non-trivial 3-torsion points of $J_X$ correspond bijectively to ways of representing $f(x)$ as $f(x) = P(x)^2-\lambda H(x)^3$, where $P(x)$ is a polynomial of degree at most 3, $\lambda$ is a non-zero element of $K$, and $H(x)$ is a monic polynomial of degree at most $2$. Given a pair $(P(x), H(x))$, a degree-0 divisor representing the corresponding $3$-torsion point can be taken to be $\frac{1}{3} \operatorname{div}(y+P(x))$.
\end{lemma}

\begin{remark}
    When $\deg H(x) = 2$, we can be more concrete: writing $H(x) = x^2+c_1x + c_0$ and $P(x) = b_3x^3+b_2x^2+b_1x+b_0$, a degree-0 divisor representing the corresponding 3-torsion point can be taken to be
\[
(x_1, y_1) + (x_2, y_2) -(\infty_+)-(\infty_-),
\]
where $x_1, x_2$ are the roots of $H(x)$ and $y_i = -P(x_i)$.
\end{remark}

\begin{remark}
    Note that \cite[Proposition 4.1]{MR3925491} claims a similar bijection using only monic polynomials $H(x)$ of exact degree 2. However, this does not seem to be correct: for example, for the curve $X : y^2=x^6+1$, there are only $74$ solutions (over $\overline{\Q})$ to the equation $x^6+1 = (b_3x^3+b_2x^2+b_1x+b_0)^2-\lambda(x^2+c_1x+c_0)^3$, which is incompatible with the fact that the Jacobian of $X$ has $80$ non-trivial $3$-torsion points. In fact, there are six more decompositions $P(x)^2-\lambda H(x)^3=x^6+1$ with $H(x)$ monic of degree $\leq 1$, four with $\deg H(x)=1$ and two with $\deg H(x)=0$. The same problem was also pointed out and corrected in the very recent preprint \cite[Appendix A]{spencer2024wildconductorexponentstrigonal}.
\end{remark}

We will use a different parametrisation, which is more suitable for our applications.
\begin{lemma}\label{lemma: 3-to-1 parametrisation}
 There is a 3-to-1 map
 \[
 \left\{ (P, Q) : \begin{array}{c} P \text{ of degree }\leq 3 \\ Q \text{ of degree } \leq 2 \\ f(x) = P^2-Q^3 \end{array} \right\} \longrightarrow J_X[3] \setminus \{0\}.
 \]
 Moreover, if $(P, Q)$ maps to the 3-torsion point $R$, then $(-P, Q)$ maps to $-R$.
\end{lemma}
\begin{proof}
 Follows from the previous lemma: for every decomposition $f(x) = P(x)^2-\lambda H(x)^3$ as in Lemma \ref{lemma: 3-torsion points} there are precisely three polynomials $Q(x)$ that satisfy $-Q(x)^3 = -\lambda H(x)^3$. For the last statement, the claim is that $\frac{1}{3} \operatorname{div}(y+P(x))$ and $\frac{1}{3} \operatorname{div}(y-P(x))$ represent torsion points whose sum is zero. In fact, their sum is $\frac{1}{3}\operatorname{div}(y^2-P(x)^2) = \frac{1}{3} \operatorname{div} (-Q(x)^3) = \operatorname{div} Q(x)$, which is a principal divisor and therefore represents the zero of $J_X$.
\end{proof}

Let $R$ be a point of order 3 on $J_X$. By definition, if we consider $R$ as a divisor of degree $0$ on $X$, then $3R$ is principal, so we have $3R = \operatorname{div} \varphi$ for some function $\varphi$ on $X$. If $R$ corresponds to the pair $(P, Q)$ as in Lemma \ref{lemma: 3-to-1 parametrisation}, the function $\varphi$ can be taken to be $y+P(x)$.
The $3$-torsion point defines a degree-3 étale cover $Y$ of $X$, an affine model of which is given by
\begin{equation}\label{eq: equations for the étale cover}
    \begin{cases}
t^3 = \varphi = y+P(x) \\
y^2 = f(x).
\end{cases}
\end{equation}

Equivalently, the function field of $Y$ is generated over the function field $K(X)=K(x,y)$ by the element $t$, with minimal polynomial $\mu_t(z) = z^3 -(y + P(x)) \in K(x,y)[z]$. Note that the converse also holds: any cyclic degree-3 étale cover of $X$ is the pull-back to $X$ of a degree-3 cover of $J_X$, hence it corresponds to a $3$-torsion point of $J_X$.
The following lemma describes the dependence of $Y$ on the choice of $(P, Q)$:
\begin{lemma}\label{lemma: -P gives the same cover}
    With notation as above,
    \begin{enumerate}
        \item $Y$ only depends on $P$;
        \item the pairs $(P, Q)$ and $(-P, Q)$ give rise to the same cover $Y$;
        \item the curve $Y$ has genus 4.
    \end{enumerate}
\end{lemma}
\begin{proof}
    Part (1) is immediate from the explicit equations \eqref{eq: equations for the étale cover}. As for (2), denote by $Y_+$ the cover \eqref{eq: equations for the étale cover} for the pair $(P, Q)$ and by $Y_-$ its analogue for the pair $(-P, Q)$. Let $t_+, t_-$ be the $t$-coordinates on $Y_+, Y_-$ respectively. An isomorphism $Y_+ \to Y_-$ is given by the map $(x, y, t_+) \mapsto (x, -y, -t_+)$. Finally, the genus of $Y$ can be computed using the Riemann--Hurwitz formula: since the map $Y \to X$ is unramified and of degree 3, we have
    \[
    2g(Y)-2 = 3(2g(X)-2) = 6 \Rightarrow g(Y)=4.
    \]
\end{proof}

Next, we describe a subgroup of $\operatorname{Aut}(Y)$ coming from the geometry of the construction.
    
\begin{lemma}\label{lemma: lift hyperelliptic involution} The following hold:
\begin{enumerate}
    \item The hyperelliptic involution $\iota_X$ of $X$ lifts to an automorphism $\iota_Y$ of $Y$, defined at the level of function fields by
 \[
 \iota_Y(x) = x, \quad \iota_Y(y) = -y, \quad \iota_Y(t)=Q(x)/t.
 \]
    \item The curve $Y_{\overline{K}}$ has an automorphism $g$ of order 3, defined at the level of function fields by $(x, y, t) \mapsto (x, y, \zeta_3 t)$, with $\zeta_3$ a primitive third root of unity. This automorphism generates the group $\operatorname{Aut}(Y_{\overline K}/X_{\overline K})$ of $\overline{K}$-automorphisms of the cover $Y \to X$.
    \item The subgroup of $\operatorname{Aut}(Y_{\overline K})$ generated by $\iota_Y$ and $\orderThreeAuto$ is isomorphic to $S_3$ and defined over $K$. This makes $Y_{\overline K}$ into a (ramified) Galois $S_3$-cover of $\mathbb{P}_1$. The original curve $X_{\overline{K}}$ corresponds to the unqiue subgroup of order 3, which is generated by $\orderThreeAuto$.
\end{enumerate}
 
\end{lemma}
\begin{proof}
Parts (2) and (3) are immediate. We show (1).
To check that $\iota_Y$ is well-defined, we verify that $\iota_Y(t)$ is a root of $\iota_X( \mu_t(z))$, that is, we check $\iota_Y(t)^3 = \iota_X(\varphi)$: indeed, one has
\[
\iota_Y(t)^3 = \frac{Q(x)^3}{t^3} = \frac{P(x)^2-f(x)}{y+P(x)} = \frac{P(x)^2-y^2}{y+P(x)} = P(x)-y = \iota_X(P(x)+y) = \iota_X(\varphi).
\]
\end{proof}

 The three elements of order 2, namely $\iota_Y \cdot \orderThreeAuto^i$ for $i=0,1,2$, give rise to quotient maps $Y \to Y / \langle \iota_Y \orderThreeAuto^i \rangle$. We now write down explicitly the quotient corresponding to $\iota_Y$.
\begin{lemma}
 The quotient $Y/\langle \iota_Y \rangle$ is a curve $E$ of genus 1 with function field $K(x, w)$, where $w$ is the function $t+\iota_Y(t) = t + \frac{Q(x)}{t}$.
\end{lemma}
\begin{proof}
 We compute the subfield $K(x, y, t)$ fixed by $\iota_Y$, which gives the function field of $Y/\langle \iota_Y \rangle$. It is clear that $x$ and $w$ are fixed by $\iota_Y$. Moreover, the degree $[K(x,y,t) : K(x,w)]$ is at most 2. Indeed, $t$ satisfies the equation $t^2-wt+Q(x) = 0$, and we have $K(x, w)(t) = K(x, y, t)$ since $y=t^3-P(x) \in K(x, w)(t)$. Finally, Galois theory shows that $[K(x,y,t) : K(x,y,t)^{\langle \iota_Y \rangle}]=2$, so $K(x,y,t)^{\langle \iota_Y \rangle}=K(x, w)$ as claimed. To check that $E$ has genus 1 we compute the ramification of the degree-2 quotient map $\phi: Y \to E$. The ramification points are the fixed points of $\iota_Y$, which are given by 
 \[
 \begin{cases}
 y=0 \\ 
 t^2=Q(x) \\
 y^2 = f(x) \\
 t^3 = y + P(x).
 \end{cases}
 \]
The solutions are given by $y=0, f(x)=0, t=P(x)/Q(x)$, a set of six points corresponding bijectively to the roots of $f(x)$. Note that $P(x)/Q(x)$ is finite for values of $x$ such that $f(x)=0$, since $f(x_0)=Q(x_0)=0$ implies that also $P(x_0)$ vanishes, in which case $x_0$ is (at least) a double root of $f(x) = P(x)^2 - Q(x)^3$, contradiction.

Each of these six points has ramification index $2$ since $\phi$ has degree 2. 
The Riemann-Hurwitz formula gives
 \[
 6 = 2g(Y)-2 = 2(2g(E)-2) + 6,
 \]
 which implies $g(E)=1$.
\end{proof}

We can now write down an equation for $E$: the functions $x$ and $w=t+Q(x)/t=t + \iota_Y(t)$ satisfy
\[
w^3 = (t+\iota_Y(t))^3 = t^3 + \iota_Y(t^3) + 3t \iota_Y(t) (t+\iota_Y(t)) = 2P(x) + 3Q(x)w,
\]
and one checks easily that 
\begin{equation}\label{eq: affine equation for E}
w^3 - 3Q(x)w - 2P(x)=0
\end{equation}
is an affine equation for $E$.

\begin{remark}
    Replacing $\iota_Y$ with any of the two other involutions produces an equation like \eqref{eq: affine equation for E} in which $Q(x)$ is replaced by $\theta Q(x)$, where $\theta$ is a primitive third root of unity. The resulting elliptic curves are all isomorphic. 
\end{remark}

Now fix a point $q_0 \in X(\overline{K})$ distinct from the six Weierstrass points of $X$. 
There are three points $p_1, p_2, p_3$ in $Y(\overline{K})$ lying over $q_0$; they all have the same $x$-coordinate, equal to the one of $q_0$. Notice that no two of these three points are exchanged by $\iota_Y$: if they were, $q_0$ would be fixed by $\iota_X$, hence would be a Weierstrass point.
The images of $p_1, p_2, p_3$ under the map $Y \to Y/\langle \iota_Y \rangle$ are therefore all distinct, hence they are the three points on \eqref{eq: affine equation for E} with $x=x(q_0)$.

We summarise the above discussion in the following proposition, where we denote by $\phi$ the quotient map $\phi : Y \to E=Y/\langle \iota_Y \rangle$.

\begin{proposition}\label{prop: properties of the construction}
Let $P(x) \in K[x], Q(x) \in K[x]$ be polynomials of degrees at most $3, 2$ respectively. Suppose that $f(x) = P(x)^2-Q(x)^3$ is a polynomial of degree 6 without repeated roots.
Let $X$ be the smooth projective curve of genus 2 with affine equation $y^2=f(x)=P(x)^2-Q(x)^3$ and denote by $q_0$ a non-special (that is, non-Weierstrass) point of $X$.

\begin{enumerate}
 \item The datum of $(P, Q)$ defines a 3-torsion point on the Jacobian $J_X$ of $X$. Let $\pi : Y \to X$ be the corresponding étale cover of degree 3. The inverse image $\pi^{-1}(q_0)$ of $q_0$
 consists of 3 distinct points $p_1, p_2, p_3$. 
 \item The hyperelliptic involution of $X$ lifts (non-uniquely) to an involution $\iota_Y$ of $Y$, as in Lemma \ref{lemma: lift hyperelliptic involution}.
 \item Let $\orderThreeAuto$ be a generator of $\operatorname{Aut}(Y_{\overline{K}}/X_{\overline{K}}) \cong \mathbb{Z}/3\mathbb{Z}$. The subgroup $\langle \iota_Y, \orderThreeAuto \rangle$ of $\operatorname{Aut}(Y_{\overline{K}})$ is defined over $K$. The group $\langle \iota_Y, \orderThreeAuto \rangle$ is isomorphic to $S_3$ and the quotient $Y / \langle \iota_Y, \orderThreeAuto \rangle$ is isomorphic to $\mathbb{P}_1$. 
 \item The quotient map $\phi : Y \to Y/\langle \iota_Y \rangle$ has degree 2, and the quotient curve $E \colonequals Y/\langle \iota_Y \rangle$ has genus 1. 
 \item An affine plane model for $E$ is given by \eqref{eq: affine equation for E}. 
 \item The images of $p_1, p_2, p_3$ in $E$ are the three points with $x$-coordinate equal to the $x$-coordinate of $q_0$.
\end{enumerate}
\end{proposition}

Finally, the above construction can be run in families (equivalently, one can take $K$ as the function field of a variety instead of a number field). For future reference, we set some notation related to the construction just described:
\begin{setup}\label{setup}
Let $\tilde{X} \to \ModuliSpace$ be a family of smooth projective curves of genus $2$ defined over a number field $k$ and let $R$ be a $3$-torsion section of $\Jac(\tilde{X}) \to \ModuliSpace$.
At least Zariski-locally on $\ModuliSpace$, the family $\tilde{X}_\ModuliParameter$ is given by an affine hyperelliptic model $y^2 = f_\ModuliParameter(x)$, and we assume that $R$ is represented by a pair $(P_\ModuliParameter, Q_\ModuliParameter)$ such $f_\ModuliParameter(x) = P_\ModuliParameter(x)^2 - Q_\ModuliParameter(x)^3$, as in \Cref{lemma: 3-to-1 parametrisation} (achieving this may require passing to a degree-3 cover of the base $\ModuliSpace$). 
Fix moreover a generically non-special section $q_\ModuliParameter$ of $\tilde{X} \to \ModuliSpace$. To this data, we attach:
\begin{enumerate}
    \item the open curve $X_\ModuliParameter = \tilde{X}_\ModuliParameter \setminus \{q_\ModuliParameter\}$ and the `point at infinity' $q_\ModuliParameter$;
  
    \item the $3$-to-$1$ cyclic étale cover $\pi : \tilde{Y}_\ModuliParameter \to \tilde{X}_\ModuliParameter$ corresponding to the $3$-torsion point $R_\ModuliParameter$, see \eqref{eq: equations for the étale cover};
    \item a generator $\orderThreeAuto$ of the Galois group of $\tilde{Y}_\ModuliParameter \to \tilde{X}_\ModuliParameter$ (we assume for simplicity that $k$ contains a third root of unity, so that $\orderThreeAuto$ is defined over $k$);
    \item the three `points at infinity' $\{p_{1, \ModuliParameter}, p_{2, \ModuliParameter}, p_{3, \ModuliParameter}\} = \pi^{-1}(q_\ModuliParameter)$ of $\tilde{Y}_\ModuliParameter$, numbered cyclically in such a way that $p_{2, \ModuliParameter}=\orderThreeAuto (p_{1, \ModuliParameter})$, $p_{3, \ModuliParameter}=\orderThreeAuto (p_{2, \ModuliParameter})$ and $p_{1, \ModuliParameter} = \orderThreeAuto (p_{3, \ModuliParameter})$;
    \item the family of genus-1 curves $E_\ModuliParameter \to \ModuliSpace$ given by Equation \eqref{eq: affine equation for E};
    \item the images $\phi(p_{1,\ModuliParameter}), \phi(p_{2,\ModuliParameter}), \phi(p_{3,\ModuliParameter})$ of the $p_{i, \ModuliParameter}$ in $E_\ModuliParameter = \tilde{Y}_\ModuliParameter / \langle \iota_{Y_{\ModuliParameter}} \rangle$. In the model \eqref{eq: affine equation for E}, these are the points with $x$-coordinate equal to $x(q_\ModuliParameter)$.
\end{enumerate}
We will often suppress the subscript $\ModuliParameter$, in particular when the base $\ModuliSpace$ is a single point, and write simply $X, Y, E, p_i$, etc.

    The family of genus-1 curves $E \to \ModuliSpace$ becomes an elliptic scheme by either replacing $E$ with its Jacobian, or by extending the base and fixing a section $e : \ModuliSpace \to E$ that gives the zero of the group law. By abuse of notation, we will often denote either of these elliptic schemes by the same symbol $E \to \ModuliSpace$ (up to replacing $\ModuliSpace$ by a finite cover, they become isomorphic). The differences $\phi(p_{2,\ModuliParameter})-\phi(p_{1,\ModuliParameter})$, $\phi(p_{3,\ModuliParameter})-\phi(p_{2,\ModuliParameter})$, and $\phi(p_{1,\ModuliParameter})-\phi(p_{3,\ModuliParameter})$ make sense both as elements of $\Jac(E)$ and as elements of $E$, independent of the choice of the origin of the group law; note that their sum is zero. Dropping the subscript $t$, we will denote by 
    \[
    \begin{array}{cccc}
    \sigma : & \ModuliSpace & \to & E^2 \\
    & t & \mapsto & (\phi(p_1)-\phi(p_2), \phi(p_2)-\phi(p_3))
    \end{array}
    \]
    the section of the doubly-elliptic scheme thus constructed (see also \Cref{lemma: projection to E^2} below).
\end{setup}

\subsection{The Jacobian of the degree-3 étale cover}
We now investigate the structure of $\Jac(Y)$.

\begin{lemma}\label{lemma: projection to E^2} Notation as in \Cref{setup}.
The following hold:
\begin{enumerate}
    \item The natural map $\phi : Y \to Y/\langle \iota_Y \rangle = E$ and the composition $\phi \circ \orderThreeAuto$ induce homomorphisms $\Jac(Y) \to \Jac(E)$ that we still denote by $\phi, \phi \circ \orderThreeAuto$. The homomorphism
    \[
    \begin{array}{cccc}
    \Phi : & \Jac(Y) & \to & \Jac(E)^2 \\
    & Q  & \mapsto & (\phi(Q), \phi \circ \orderThreeAuto(Q))
    \end{array}
    \]
    gives an isogeny between $\ker \left(\Jac(Y) \to \Jac(X)\right)$ and $\Jac(E)^2$.
    \item Let $Q = [(p_1)-(p_2)] \in \Jac(Y)$. The image of $Q$ in $E^2$ via $\Phi$ is $( \phi(p_1)-\phi(p_2), \phi(p_2)-\phi(p_3) )$.
\end{enumerate}
\end{lemma}
\begin{proof}
Note that $\phi$ is the quotient by $\iota_Y$, so we can identify $\phi$ with the map $Q \mapsto Q + \iota_Y(Q)$. 

There is a canonical map $\Jac(X) \to \Jac(Y)$ (pull-back of divisors). If $\sum_i n_i(x_i)$ is a point in $\Jac(X)$, its image in $\Jac(Y)$ is of the form $\sum_i \sum_{j=0}^2 n_i(\orderThreeAuto^j\tilde{x}_i)$, where $\tilde{x}_i$ is any point of $Y$ lying over $x_i$. Applying $\phi$ we get $\sum_i \sum_{j=0}^2 n_i(\orderThreeAuto^j\tilde{x}_i + \iota_Y\orderThreeAuto^j \tilde{x}_i)$, which is a principal divisor: indeed, it is invariant under the group $\langle \iota_Y, \orderThreeAuto \rangle$, so it is pulled back from $Y/\langle \iota_Y, \orderThreeAuto \rangle \cong \mathbb{P}^1$, where every divisor of degree  $0$ is principal. This shows that $\Phi$ is trivial on the canonical image of $\Jac(X)$ in $\Jac(Y)$.
\begin{enumerate}
    \item The map $\Phi$ is clearly a homomorphism of abelian varieties. To see that it is surjective, 
    we consider the three maps $\pi_1 := \phi, \pi_2 := \phi \circ \orderThreeAuto$ and $\pi_3 := \phi \circ \orderThreeAuto^2$: the group $S_3$ acts in the natural way on these, and $\pi_1+\pi_2+\pi_3=0$. We thus get a map $\Jac(Y)\to E^3$ with image in $\{(x,y,z) : x+y+z=0\} \cong E^2$. The surjectivity of this map is equivalent to that of $\Phi$. Suppose this map is \textit{not} surjective, hence $a\pi_1=b\pi_2$ for some $a,b \in \operatorname{End}(E)$ not both zero (say $b \neq 0$). Acting with $S_3$ we find $a\pi_1=b\pi_3$, hence by difference $b(\pi_2-\pi_3)=0$, and therefore $\pi_2=\pi_3$. But this gives $\pi_1=\pi_2=\pi_3$, and from $0 =\pi_1+\pi_2+\pi_3=3\pi_1$ we obtain that $\pi_1=\phi$ is the zero map, contradiction.
    
    We have already shown that $\Phi$ is trivial on the image of $\Jac(X)$ in $\Jac(Y)$. Note that $\ker(\Jac(Y) \to \Jac(X))$ and $\Jac(X)$ intersect in a finite set  inside $\Jac(Y)$.
    Finally, observe that the dimension of $\ker(\Jac(Y) \to \Jac(X))$ is $\dim \Jac(Y) - \dim \Jac(X) = \text{genus}(Y)-\text{genus}(X)=2=\dim E^2$ and recall that  any surjective homomorphism between abelian varieties of the same dimension is an isogeny.
    \item Clear from the definitions: the action of $\orderThreeAuto$ sends $p_1, p_2$ to $p_2, p_3$ respectively, hence $\phi(\orderThreeAuto(Q)) = \phi( (p_2)-(p_3)) = \phi(p_2)-\phi(p_3)$.
\end{enumerate}
\end{proof}

\begin{remark}
    The decomposition (up to isogeny) of the Jacobian of $\tilde{Y}$ as $J_X \times E^2$ also appears in \cite{Lomb} as a special case of the decomposition of jacobians via Galois covers (see in particular Table 7 of \cite{Lomb}).
\end{remark}

\begin{lemma}\label{lemma: torsion on J_Y iff both torsion on E}
    Notation as in \Cref{setup}. Replacing $\ModuliSpace$ with a finite cover of it, we have rational sections $p_i = p_{i, \ModuliParameter} : \ModuliSpace \to \tilde{Y}_\ModuliParameter$ and $\phi(p_i) = \phi(p_{i,\ModuliParameter}) : \ModuliSpace \to E_\ModuliParameter$. The following are equivalent:
    \begin{enumerate}
        \item the sections $p_1-p_2$ and $p_3-p_2$ of $\Jac(\tilde{Y}_\ModuliParameter)$ are both torsion;
        \item the section $p_1-p_2$ of $\Jac(\tilde{Y}_\ModuliParameter)$ is torsion;
        \item the differences $\phi(p_1)-\phi(p_2), \phi(p_3)-\phi(p_2)$ are both torsion on $E_\ModuliParameter$.
    \end{enumerate}
    We note explicitly that these conditions are \textit{not} equivalent to the single section $\phi(p_1)-\phi(p_2)$ being torsion on $E_\ModuliParameter$.
\end{lemma}
 \begin{proof}
 Clearly, (1) implies (2). To show that (2) implies (1), notice that the generator $\orderThreeAuto$ of $\operatorname{Gal}(\tilde{Y}_\ModuliParameter/\tilde{X}_\ModuliParameter)$, acting on $\Jac(\tilde{Y}_\ModuliParameter)$, sends the torsion point $p_1-p_2$ to $p_2-p_3$, which is then necessarily torsion (hence so is $p_3-p_2$). It is also clear that (1) implies (3), because the morphism of abelian varieties $\Jac(Y)\to E$ sends torsion points to torsion points. Finally, Lemma \ref{lemma: projection to E^2} shows that (3) implies (2), because any inverse image of a torsion point under an isogeny is a torsion point.
 \end{proof}

\begin{remark}\label{rmk: 40 different curves}
    The above construction yields a map from the set of curves $X$ of genus $2$, together with a 3-torsion point $R$ of their Jacobian, to the set of elliptic curves. Since $\Jac(X)$ has 80 points of exact order 3, it is natural to wonder how the corresponding elliptic curves are related. Denote by $E(X, R)$ the elliptic curve associated with $X$ and a choice of $R \in \Jac(X)[3] \setminus \{0\}$. We claim that, generically, the curves $E(X, R)$ and $E(X, R')$ are geometrically isogenous if and only if $R'= \pm R$. If $R' = \pm R$, then the corresponding curves are even isomorphic (Lemma \ref{lemma: -P gives the same cover}). For the other implication, it suffices to give an example where all 40 elliptic curves obtained from a single $X$ are pairwise geometrically non-isogenous.

To find such an example, we rely on \cite{MR3893190}: in that paper, the authors explicitly show that the \textit{{B}urkhardt quartic} is the moduli space of abelian surfaces with full 3-level torsion and give rational parametrisations of it. Using their results, one can easily find curves of genus 2 defined over $\Q$ whose Jacobian has full 3-torsion over a small extension of $\Q$. For example, taking $t_1=2, t_2=-1, t_3=3$ in \cite[Theorem 4.2]{MR3893190} gives a curve $X/\Q$ for which $\Jac(X)$ has full 3-torsion over a sextic field (the curve is explicit, but its equation has large height and we do not reproduce it here). We then explicitly compute the $j$-invariants of the 80 elliptic curves attached to $X$ and check, using the techniques of \cite{MR3882288}, that they correspond to 40 pairwise geometrically non-isogenous elliptic curves. An implementation of this calculation is included in the online repository \cite{Computations}.
\end{remark}

\subsection{Models of genus-2 curves}\label{sect: models of genus 2 curves}

We first briefly discuss what we mean by \textit{integral point} of a curve of genus $2$ with a marked point.
Let $X/K$ be a (smooth projective) curve of genus 2 and let $q \in X(K)$ be a non-special $K$-rational point. Let $S$ be a finite set of places of $K$, containing those where $X$ has bad reduction. There is a unique minimal (smooth, regular) model $\mathcal{X}$ of $X$ over $\mathcal{O}_{K, S}$, and (enlarging $S$ if necessary) $q$ extends to a $\mathcal{O}_{K, S}$-point $\mathcal{Q}$ of $\mathcal{X}$. We are then interested in the $\mathcal{O}_{K, S}$-points of $\mathcal{X} \setminus \mathcal{Q}$.

In our case of genus 2 curves, we can describe the situation much more concretely. Extending $S$ if necessary, we can assume \cite{MR4469240} that $\mathcal{X}$ is given by a Weierstrass model $y^2=f(x)$, where $f(x)$ is a separable polynomial of degree $5$ or $6$ with discriminant in $\mathcal{O}_{K,S}^\times$. Up to changing model if necessary, the point $\mathcal{Q}$ (the extension of $q$) corresponds to a solution $(x_q,y_q)$ in $\mathcal{O}_{K,S}$-integers of $y^2=f(x)$. We are then looking for the $\mathcal{O}_{K,S}$-points of $y^2=f(x)$ that do not reduce to $(x_q, y_q)$ modulo any prime of $\mathcal{O}_{K, S}$. From this point of view, the uniqueness of the integral model can be seen as follows: any two hyperelliptic models $y^2=f_1(x)$ and $y^2=f_2(x)$ are related by a M\"obius transformation $M$, and the requirement that both $f_1(x), f_2(x)$ have discriminant prime to $S$ implies that $M$ is invertible over $\mathcal{O}_{K, S}$.

This describes the situation for a curve $X$ defined over a number field $K$. We now discuss the situation for curves over $\overline{\Q}$.

Continuing from \Cref{rmk: descent does not matter}, we point out that many of the subtleties related to moduli spaces are well-understood in genus 2 (see for example \cite{MR2990029, MR3471117} for an extended discussion in the case of hyperelliptic curves of arbitrary genus). On the one hand, every curve of genus 2 admits a hyperelliptic model over any field where it is defined. On the other hand, there is in general no minimal unique field of definition: technically, in characteristic 0 a curve of genus 2 cannot always be defined over its field of moduli, and this is measured by Mestre's obstruction \cite{MR1106431, MR2181874}. 
For any given genus 2 curve $X/\overline{\mathbb{Q}}$ and any given number field $K$, one can decide effectively whether $X$ admits a descent $X_0$ to $K$, and in this case, whether the point at infinity on $X$ corresponds to a $K$-rational point of $X_0$. 

If a descent $X_0/K$ exists, all other descents (called the \textit{twists} of $X_0$) are parametrised by $H^1(\operatorname{Gal}(\overline{K}/K), \operatorname{Aut}(X_{0, \overline{\Q}}))$. This is an infinite set, but one can restrict to a finite set by asking, for example, for all twists having good reduction outside a given finite set of places of $K$. With this restriction, the set of twists then becomes computable: one can enumerate all Galois cocycles that are unramified outside a given set of places, and given a cohomology class in $H^1(\operatorname{Gal}(\overline{K}/K), \operatorname{Aut}(X_{0, \overline{\Q}}))$, represented by a cocycle, the corresponding twist can be computed explicitly \cite{MR3906177, Card}. Note also that -- since we are interested in integral points -- one could work with the subgroup of $\operatorname{Aut}(X_{0, \overline{\Q}})$ fixing the given point at infinity, or instead compute all possible twists and determine for which of these the point at infinity is $K$-rational.
Finally, given two descents $X_1/K_1$ and $X_2/K_2$, where $K_1, K_2$ are number fields, one can compute a minimal common extension of $K_1, K_2$ over which $X_1, X_2$ become isomorphic. Combining all these ingredients, it is not hard to see that our methods (when they apply) give effectivity for the $\mathcal{O}_{K,S}$-integral point of \textit{any} model of $X$, over \textit{any} number field (containing a field of definition) and for \textit{any} finite set of places $S$.

Finally, fix a genus-2 curve $X/\overline{\Q}$ to which our method applies and a descent $X_0/K$ over a \textit{fixed} number field $K$. Also fix a finite set $S$ of places of $K$, containing those where $X_0$ has bad reduction. We can then determine all integral points on all twists of $X_0/K$ having good reduction outside of $S$: there is only a finite number of such twists, considered as curves over $K$, and as explained above, these can be effectively determined. For each model, a suitable application of Bilu's criterion allows us to find all $\mathcal{O}_{K, S}$-points.

\medskip

We conclude this discussion by describing (singular) models for smooth projective curves of genus $2$ with a marked point, which are alternative to the more usual hyperelliptic ones.
\begin{lemma}\label{lemma: singular model with one point at infinity}
    Let $\tilde{X}$ be a smooth projective curve of genus $2$ and let $q_0$ be a non-special point of $\tilde{X}$. Write $X= \tilde{X} \setminus \{q_0\}$. There is a morphism $f : \tilde{X} \to \mathbb{P}_2$, birational onto its image, that restricts to a morphism $X \to \mathbb{A}^2$ whose image is a plane quartic with precisely one singular point, which is a node. The point $f(q_0)$ is the unique point at infinity of $f(X)$.
\end{lemma}

\begin{proof}
Consider the linear system $L(3q_0)$ of rational functions on $\tilde X$ with at most a triple pole at $q_0$. Its global sections are generated by $1,z$, where $z$ is a suitable rational function of degree $3$ (since $q_0$ is non-special). Then $L(4q_0)$ is generated by $1,z,w$, where $w$ has degree $4$. Since $z,w$ have coprime degrees, they generate the function field of $X$. Since they have a unique pole, they provide a map $X\to\A^2$ with a single point at infinity.  The extended map $(1:z:w)$ to $\P_2$ provides the projective plane model, with coordinate functions   $1/w$ and $z/w$ at infinity, both vanishing there. Since $z/w$ has a simple zero at infinity,
the curve is smooth there. 

Let us show that the irreducible polynomial equation $f(z,w)=0$ relating $z,w$ is of total degree $4$. First of all, it has degree $4$ in $z$ and $3$ in $w$ (because $z,w$ have degrees respectively $3,4$ as functions on $\tilde X$). Now, the space $L(15q_0)$ contains the $15$ functions $1,z,z^2,z^3,z^4,z^5, w, wz, wz^2, wz^3, w^2, w^2z, w^2z^2, w^3, w^3z$. By Riemann-Roch they must be linearly dependent, giving a nontrivial polynomial relation $g(z,w)=0$ where $g$ is a linear combination of the 15 monomials listed above. This $g$ must be divisible by $f$: since $f$ has no term $z^5$ and $g$ has no other term of degree greater than $4$, this shows that $f$ has degree at most (hence exactly) $4$, as stated\footnote{It seems that in this case the usual analysis using the spaces $L(mq)$ needs to be complemented by supplementary considerations, such as those involving $f, g$ that we used here.}.

This quartic model will necessarily have a single singular (nodal) point in $\A^2$ by the genus formula. 

Up to translation, we may and do assume that $(0,0)$ lies on the image of $X$ in $\mathbb{A}^2$ and that it is the singular point. The equation $f(z, w)=0$ then has the form
\[
a_9 z^2 + a_6 z^3 + a_3 z^4 + a_8 wz + a_5 wz^2 + a_2 wz^3 + a_7 w^2 + a_4 w^2 z + a_1 w^2z^2 + b_3 w^3 + a_0 w^3z = 0.
\]
Denote by $[Z:W : \LastCoord]$ projective coordinates on $\mathbb{P}_2$. We know that there is a unique point at infinity, with coordinates $[0:1:0]$, so the homogeneous part of degree 4 of the above equation,
\[
a_3 Z^4 + a_2 WZ^3 + a_1 W^2Z^2 + a_0 W^3Z = 0,
\]
must have a quadruple root at $z=0$, which forces $a_2=a_1=a_0=0$. We are left with the equation
\[
a_9 z^2 + a_6 z^3 + a_3 z^4 + a_8 wz + a_5 wz^2 + a_7 w^2 + a_4 w^2 z  + b_3 w^3 = 0.
\]
Rescaling $w$ and $z$ we may assume $a_3=b_3=1$, and a further substitution $w \mapsto w - \alpha z$ for a suitable $\alpha$ allows us to impose that the coefficient of $w^2z$ vanishes. Thus, every curve of genus 2 may be represented by a singular model of the form
\[
a_9 z^2 + a_6 z^3 + z^4 + a_8 wz + a_5 wz^2 + a_7 w^2   + w^3 = 0
\]
with $5$ free parameters. Finally, notice that rescaling $z \mapsto \lambda^3 z, w \mapsto \lambda^4 w$ and then dividing the whole equation by $\lambda^{12}$ we get
\[
a_9 \lambda^{-6} z^2 + a_6 \lambda^{-3} z^3 + z^4 + a_8 \lambda^{-5} wz + a_5 \lambda^{-2} wz^2 +  a_7 \lambda^{-4} w^2   + w^3 = 0.
\]
By choosing $\lambda$ suitably, we can ensure that the coefficient of $z^2$ is either $0$ or $1$. 
In the second case, by further rescaling, we can assume that $a_7=1$ or $0$; finally, if also $a_7=0$, then necessarily $a_8 \neq 0$ (otherwise the origin would not be a node), and we can rescale to set $a_8=1$. Thus, we end up with the three families
\begin{equation}\label{eq: general curve of genus 2 with a marked point}
a_9 z^2 + a_6  z^3 + z^4 + a_8  wz + wz^2 + a_7 w^2   + w^3 = 0,
\end{equation}
\[
a_6 z^3 + z^4 + a_8 wz + a_5 wz^2 + w^2 + w^3 = 0,
\]
\[
a_6 z^3 + z^4 + wz + a_5 wz^2 + w^3 = 0.
\]
The first, and most general, of these depends on four free parameters, which agrees with the dimension of the corresponding moduli space of genus-2 curves with a marked point.
\end{proof}

\begin{remark}[Hyperelliptic model of \eqref{eq: general curve of genus 2 with a marked point}]
Consider the plane curve given in Equation \eqref{eq: general curve of genus 2 with a marked point}. The general line through the origin of the $(w, z)$-plane (which we write as $z=m w$) meets this curve in four points, including the origin with multiplicity two. Replacing $z=mw$ in the above equation we find
\[
a_9 m^2w^2 + a_6  m^3w^3 + m^4w^4 + a_8  mw^2 + m^2w^3 + a_7 w^2   + w^3 = 0,
\]
and dividing through by $w^2$ we obtain
\[
a_9 m^2 + a_6  m^3w + m^4w^2 + a_8  m + m^2w + a_7 + w = 0.
\]
This equation realizes our curve (up to birational equivalence) as a degree-2 cover of the projective line with coordinate $m$. To write it as a usual hyperelliptic model, it suffices to write the discriminant of the degree-2 equation for $w$ as a function of $m$:
\[
\Delta = (m^2+1 + a_6  m^3)^2 - 4m^4(a_8  m + a_9 m^2 + a_7).
\]
Thus, a hyperelliptic model for \eqref{eq: general curve of genus 2 with a marked point} is given by
\[
z^2 = (a_6^2 - 4 a_9) m^6 + (2 a_6-4a_8) m^5 + (1- 4 a_7) m^4 + 2 a_6 m^3  + 2 m^2 + 1.
\]
\end{remark}

\section{Effectivity for curves of genus $2$}\label{sect: effectivity in genus 2}

In this section we discuss the applicability of the following criterion, due Bilu \cite{Bi}, to families of curves of genus 2:

\begin{theorem}[Bilu's criterion, first version]\label{thm: Bilu version 1} Let $X$ be a non-singular curve over a number field $F$.
    If the group of regular morphisms $X\to\G_m$ has rank $\ge 2$, the set $X(\mathcal{O}_{K,S})$ of $S$-integral
    points on $X$ can be effectively determined for every number field $K$ containing $F$ and every finite set $S$ of places of $K$.
\end{theorem}

We will express the conclusion of \Cref{thm: Bilu version 1} by saying that \textit{there is effectivity for the integral points of $X$}. \Cref{thm: Bilu version 1} may be reformulated as follows, where we write $\tilde{X}$ for the smooth projective completion of $X$ and $J_X$ for the Jacobian of $\tilde{X}$.

\begin{theorem}[Bilu's criterion, second version]\label{thm: Bilu version 2} 
Let $m := \#(\tilde{X} - X)$ be the number of points at infinity of $X$. Suppose that the subgroup of $J_X$ generated by the differences between points at infinity of $X$ has rank at most $m-3$: then there is effectivity for the integral points of $X$.
\end{theorem}

\begin{remark}
    Since the maximal possible rank of the subgroup of $J_X$ generated by differences of points at infinity is $m-1$, if the hypothesis of the previous theorem is satisfied we say that the subgroup generated by the differences between points at infinity has co-rank at least $2$.
\end{remark}

One can give still other equivalent versions of this criterion. Bilu's result is ultimately based on the following fact:

\begin{proposition}
    The solutions of the equation $f(x,y)=0$ in $S$-units $x,y$ can be effectively determined.
\end{proposition}
\begin{proof}
Follows upon combining the theory of linear forms in logarithms with the properties of Puiseux series, see Bilu's papers or \cite[Chapter 5]{BG} for an argument which avoids even appealing to Puiseux's theorem.
\end{proof}

As Bilu did in some cases (e.g., for modular curves \cite{Bi, BiI}), one may apply the criterion after taking a cover $Y \to X$ that is unramified outside the points at infinity (the Chevalley-Weil theorem shows that in order to determine the integral points on $X$ it suffices to determine the integral points on finitely many twists of $Y$). As above, we denote by $\tilde{Y}$ the smooth projective completion of $Y$ and by $J_Y$ the Jacobian of $\tilde{Y}$.
Then, as formulated explicitly by Bilu \cite[Theorem E]{Bi}, we have the following sufficient condition for the effectivity of integral points:

\begin{theorem}[Bilu's criterion, third version]\label{thm: Bilu version 3}
Notation as in \Cref{thm: Bilu version 1}. Let $\pi : Y \to X$ be an unramified cover. If the subgroup of $J_Y$ generated by the differences of pairs of points in $\tilde{Y} \setminus Y$ has co-rank at least $2$, then there is effectivity for the integral points of both $X$ and $Y$.
\end{theorem}

\begin{remark}\label{rmk: how to apply Bilu's criterion}
We will apply \Cref{thm: Bilu version 3} when $\tilde{X}$ is a curve of genus 2 and $\pi : \tilde{Y} \to \tilde{X}$ is a degree-3 étale cover. Writing $X=\tilde{X} \setminus \{q\}$ and $\pi^{-1}(q)=\{p_1, p_2, p_3\}$, the assumption of \Cref{thm: Bilu version 3} comes down to the fact that the differences $p_1-p_2, p_2-p_3$ are torsion in $\Jac(\tilde{Y})$.
\end{remark}

\subsection{The simplest case of our method}
 
We take the smooth projective curve $\tilde X$ with function field given by $v^2=f(u)$, where $f$ is a separable polynomial of degree $6$.
We define $X=\tilde X-\{q_0\}$ for a (non-special) point $q_0\in\tilde X$, to be chosen later, and let $J_X$ be the Jacobian of $\tilde X$.

As in Setup \ref{setup}, we construct an unramified\footnote{For Bilu's criterion it would suffice that $\pi$ is unramified except above $q_0$. We forget about this freedom in the present approach. The paper \cite{LP} shows that `generically' this freedom does not lead to additional room for manoeuvre.} cyclic cover $\pi:\tilde Y\to \tilde X$ of degree $3$, with points $p_1, p_2, p_3$ above $\infty=q_0\in X$.  If the Galois group of $\pi$ is generated by an element $\orderThreeAuto$ of order $3$, we may label the points so that $p_{i+1}=\orderThreeAuto(p_i)$ for $i=1, 2, 3$ (indices read modulo $3$). 
We suppose that the ground field contains a cubic root of $1$. In fact, we may take any finite extension of the ground field, or equivalently work over $\overline{\Q}$. The cover $\tilde Y$ corresponds to taking a cube root of a rational function $\varphi$ on $\tilde X$ and comes from a $3$-torsion point of $J_{\tilde{X}}$, see Section \ref{sec: details of the construction} and \eqref{eq: equations for the étale cover} in particular. There are 80 (non-trivial) three-torsion points on $J_{\tilde{X}}$, giving the same $\tilde{Y}$ in pairs (\Cref{lemma: -P gives the same cover}).

\subsubsection{Applying Bilu's criterion} The curve $Y$ has three points at infinity. To apply Bilu's criterion, in the form of \Cref{thm: Bilu version 3}, we need the points $p_i-p_j$ to be torsion in $J_Y$.

\begin{remark}
A priori, the condition that all the differences $p_i-p_j$ are torsion seems too difficult to achieve: indeed, $J_Y$ has dimension $4$, while the whole moduli space of curves of genus $2$ has dimension $3$. The choice of point at infinity gives one more degree of freedom, but this is still a priori not enough to obtain infinitely many examples. For a sense of the limited hopes of obtaining examples where Bilu's criterion applies, see the paper \cite{LP}.
\end{remark}

Despite the previous remark, there are three peculiar features in our situation that give us more freedom, as we now illustrate. We will consider the curve $Y$ to vary in a family $Y_{\ModuliParameter}$ over some moduli space $\ModuliSpace$.

\begin{remark}\label{rmk: special features of our situation}
Our construction enjoys the following properties:
\begin{enumerate}
    \item [A.] There is a surjective homomorphism $J_Y\to J_X$, so there is an isogeny $J_Y \approx J_X \times J'$, where $J' = \ker\left( J_Y\to J_X\right)$, see Lemma \ref{lemma: projection to E^2}. 
Note that $\pi\circ \orderThreeAuto=\pi$, so the differences $p_i-p_j$ map to $0$ in $J_X$. Thus, their classes in $J_Y$ belong to the 2-dimensional subvariety $J'$. This is already a big gain.

\item [B.] The differences $p_i-p_j$, seen as classes in $J_Y$, are not generically torsion (there are many ways to see this: see \S\ref{subsec: sections generically independent} for a computational approach and Proposition \ref{prop: no identically torsion section on the whole moduli space} and \Cref{lemma: torsion on J_Y iff both torsion on E} for a stronger statement). Note that, if they were, we could apply Bilu's criterion at once to the whole set of genus-2 curves.

Using the cyclic structure, it follows that $p_3-p_1$ and $p_2-p_1$ are generically independent over $\Z$ (i.e., as $\ModuliParameter$ and the curve $\tilde X_\ModuliParameter$ vary). Indeed, any linear equivalence relation $ap_1+bp_2+cp_3\sim 0$, where $a,b,c\in\Z$ satisfy $a+b+c=0$,  implies $ap_2+bp_3+cp_1\sim ap_3+bp_1+cp_2\sim 0$, and it follows that 
$(a^2+ab+b^2)(p_1-p_2) \sim (a^2+ab+b^2)(p_1-p_3) \sim 0$, so the differences $p_3-p_1, p_2-p_1$ would be torsion, contradiction.

However, the sections $p_1-p_2,p_2-p_3,p_3-p_1$ \textit{are} linearly dependent over the ring $\operatorname{End}(\Jac(\tilde{Y}_\ModuliParameter))$: indeed, $\Jac(\tilde{Y}_\ModuliParameter)$ has an automorphism $\orderThreeAuto$ of order $3$ coming from the corresponding automorphism of $\tilde Y$, and $\{p_1-p_2,p_2-p_3,p_3-p_1\}$ is an orbit, so $p_3-p_2$ is the image $\orderThreeAuto(p_2-p_1)$ of $p_2-p_1$.
In particular, if one of the differences is torsion, then any other is also torsion. Therefore, it is sufficient to ensure that one of them is torsion to satisfy the conditions of Bilu's criterion (note that $p_2-p_1, p_3-p_1$ are independent over $\Z$ as divisors). This is a further gain, and (a priori) we now have 4 or 3 degrees of freedom, depending on whether we vary $q_0$ or not, against only 2 constraints (i.e., the dimension of $J'$). 

 \item [C.] 
As described in Lemma \ref{lemma: projection to E^2}, the abelian variety $J'$ is not a general abelian surface, but rather the square of an elliptic curve.

\end{enumerate}
\end{remark}

\begin{remark}
For a fixed curve $\tilde{X}$ of genus 2, there are 80 choices of a 3-torsion point in $\Jac(\tilde{X}_\ModuliParameter)$ (Lemma \ref{lemma: 3-torsion points}). Each 3-torsion point $P$ gives us an étale cover $\tilde{Y}_P \to \tilde{X}$ and a corresponding elliptic curve $E_P$ (see \Cref{setup}). 
The curves $E_P$ are in general not isomorphic to one another: in fact, the set of their $j$-invariants generically contains $40$ distinct values, and generically the 40 elliptic curves obtained in this way are not even pairwise isogenous, see Remark \ref{rmk: 40 different curves}.
\end{remark}

The isogeny decomposition $J_Y\approx J_X\times J' \approx J_X \times E^2$ 
yields further freedom, since we may seek for instance a section which is identically torsion on one elliptic factor but not on the other (see Section \ref{sect: examples} for explicit examples of this and other phenomena), or such that its two elliptic components are identically linearly dependent, and then we only have to prescribe torsion on one of the two factors, which is a weaker constraint, allowing us to lower the dimension of the base.
We give in \Cref{subsect: pencil} an example of a 1-dimensional family based exactly on this type of construction.

\subsection{Proof of \Cref{T}} \label{SS.hints} 

Let $\ModuliSpace$ be as in the statement and let $\ModuliParameter \mapsto q_\ModuliParameter$ be a (generically) non-special algebraic section. As already noticed in \Cref{rmk: singular points}, we can replace $\tilde{X}_t$ with a smooth curve birational to it without affecting the problem of finding its integral points (notice that the point $q_t$ is assumed to be smooth, so it corresponds to a unique point on the smooth model). We can therefore assume that the curves $\tilde{X}_t$ are smooth.

For each smooth complete curve $\tilde X$ of genus $2$ (except possibly for the curves corresponding to finitely many surfaces in $M_2$), we have a finite non-empty set (with uniformly bounded cardinality) of $\ModuliParameter\in \ModuliSpace$ such that $\tilde X_\ModuliParameter$ is isomorphic to $\tilde X$, and we have a (generically non-special) point $q_\ModuliParameter\in \tilde X_\ModuliParameter$, where the coordinates of $q_\ModuliParameter$ are algebraic functions on $\ModuliSpace$.

Replacing $\ModuliSpace$ by a finite cover (again denoted by $\ModuliSpace$ for convenience) if necessary, we find ourselves in the situation of \Cref{setup}.
In particular, we can fix a $3$-torsion section of $\Jac(\tilde{X}_\ModuliParameter)$, represented as in \Cref{lemma: 3-to-1 parametrisation}, and construct a family $\tilde{Y}_\ModuliParameter \to \ModuliSpace$.
The Jacobians of the curves $\tilde Y_\ModuliParameter$ determine an abelian scheme $\Jac(\tilde{Y}_\ModuliParameter)$ over $\ModuliSpace$, and we denote by $\orderThreeAuto$ the generator of the order-3 group $\operatorname{Aut}(\tilde Y_\ModuliParameter/\tilde X_\ModuliParameter)$ and by $p_i = p_{i, \ModuliParameter}$ the three points of $\tilde Y_\ModuliParameter$ lying over $q_\ModuliParameter$.
 
Replacing $\ModuliSpace$ by a further finite cover (still denoted by $\ModuliSpace$), we obtain two sections from $\ModuliSpace$ to the abelian scheme $\Jac(\tilde{Y}_\ModuliParameter)$, i.e., the classes of the differences $p_2-p_1$, $p_3-p_1$.  As noted in Remark \ref{rmk: special features of our situation} (B), these sections are dependent over $\operatorname{End}(\Jac(\tilde{Y}_\ModuliParameter))$. If they were dependent over $\Z$, by Remark \ref{rmk: special features of our situation} (B) again 
the classes $p_2-p_1, p_3-p_1$ would be identically torsion on $\ModuliSpace$.  
If this were the case, then we could apply Bilu's criterion to all members of the family, leading to a stronger theorem. So let us assume that $p_2-p_1, p_3-p_1$ are not identically linearly dependent over $\Z$. 

\begin{remark}
    In fact, one can \textit{prove} that the classes $p_2-p_1, p_3-p_1$ cannot be simultaneously identically torsion on the whole $\ModuliSpace$, irrespective of the choice of the section $q_\ModuliParameter$. We will show this in Proposition \ref{prop: no identically torsion section on the whole moduli space}, and note that it follows from a general theorem in \cite{CMZ} that this is true except for a finite number of sections $q_\ModuliParameter$ (see \Cref{lemma: no identically torsion section with finitely many exceptions}).
\end{remark}

Denote by $\sigma_1, \sigma_2$ the sections $p_2-p_1, p_3-p_1$ of $J'$ (notation as in \Cref{rmk: special features of our situation} (A)). By a slight abuse of notation, we will denote by the same symbols the sections of $E^2$ obtained by composing $\sigma_1, \sigma_2$ with the isogeny $J' \approx E^2$ of \Cref{lemma: projection to E^2}.
Note that since $\ModuliSpace \to M_2$ is dominant and of finite degree we have $\dim \ModuliSpace = 3$.
By the theory of the \textit{Betti map} (see \cite{MR2918151}), we know that $\sigma_1$ becomes torsion over a dense set $\Sigma$, so by \Cref{rmk: special features of our situation} (B) the section $\sigma_2$ will be torsion as well over the same set $\Sigma$. The fact that a non-torsion section has a dense set of torsion specialisations is not easy to check in general (see for example \cite{ACZ}), but here the situation is simpler since the relevant abelian varieties are isogenous to squares of elliptic curves.
This case is easier and treated in particular in \cite{CMZ}.

For every point $\ModuliParameter \in \Sigma$, \Cref{thm: Bilu version 3} applies to the cover $\tilde{Y}_\ModuliParameter \to \tilde{X}_\ModuliParameter$, giving the desired effectivity.

\qed

We make some remarks about the torsion points constructed at the end of the previous proof.
\begin{remark} 
\begin{enumerate}
    \item The orders of the torsion specialisations $\sigma(\ModuliParameter)$ grow to infinity unless the section is identically torsion on the whole base $\ModuliSpace$; this last possibility is
excluded by Proposition \ref{prop: no identically torsion section on the whole moduli space}. 
\item Let $E \to \ModuliSpace$ be an isotrivial elliptic scheme. Up to an extension of the base, $E$ can be assumed constant. The assertion that a non-torsion section of $E$ has infinitely many torsion specialisations is not generally true, but remains true (with a very simple proof) provided that the section is not a non-torsion constant. In our case, we know that the elliptic family is not isotrivial, since the $j$-invariant of $E$ is non-constant on $M_2$ (see \Cref{prop: joint j-invariants} for a much stronger statement).
\end{enumerate}    
\end{remark}

\section{The section is generically non-torsion}\label{S.nt}

The arguments of \S \ref{SS.hints} show that \textit{if} a certain section of a doubly elliptic scheme, depending on the choice of the section $\ModuliParameter \mapsto q_\ModuliParameter\in \tilde X_\ModuliParameter$ giving the points at infinity on our curves, is identically torsion, \textit{then} we have effectivity for the integral points on any affine smooth curve $X_\ModuliParameter$ of genus $2$ defined by $X_\ModuliParameter:=\tilde X_\ModuliParameter-\{q_\ModuliParameter\}$. We now show that this cannot happen when the family of curves we consider is three-dimensional (more precisely, dominant onto the relevant moduli space), regardless of the choice of the section at infinity $q_\ModuliParameter$. This is another limitation of the method.

\begin{proposition}\label{prop: no identically torsion section on the whole moduli space} With notation as in \Cref{T}, there is no algebraic family $\{q_\ModuliParameter\}_{\ModuliParameter\in \ModuliSpace}$, $q_\ModuliParameter\in\tilde X_\ModuliParameter$, such that the section of the doubly elliptic scheme constructed as in Setup \ref{setup} is torsion on the whole $\ModuliSpace'$, where $\ModuliSpace'$ is a finite cover of $\ModuliSpace$ where all the data of \Cref{setup} are defined.
\end{proposition}

The paper \cite{LP} goes in a somewhat similar direction to this statement, in that it also provides limitations on the applicability of Bilu's criterion. However, there are significant differences with the present proposition. In a sense, the main result of \cite{LP} is much more general, since it concerns arbitrary covers of arbitrary curves of genus $>1$. It is also different because it starts with a {\it given} curve (rather than the whole moduli space) and proves that for a general choice of the point at infinity, Bilu's criterion does not hold. Instead, here we work only with cyclic cubic covers, but we let the (genus-2) curves vary, and we assume that the choice of the point at infinity is given.

On the other hand, \cite{LP} deals with objects defined over $\C$, whereas we can work over $\overline\Q$, so in this respect our result is stronger. The abundance of specialisation arguments in the literature might lead one to expect that what holds over $\C$ should also hold over $\overline{\Q}$; however, one should be cautious about taking such inferences for granted, as shown by notable examples such as Belyi's theorem. 

Note that assuming that the value of the section at $\ModuliParameter$ is torsion for each $\ModuliParameter\in \ModuliSpace(\overline\Q)$ amounts to assuming that it is torsion at a generic point of $\ModuliSpace$: if the generic value is not torsion, well-known arguments of Silverman (see Demianenko-Manin in the isotrivial case) prove that the set of $\ModuliParameter$ for which the specialisation is torsion has bounded height.

By appealing to a general result, we can easily prove that Proposition \ref{prop: no identically torsion section on the whole moduli space} holds up to finitely many choices of $q_\ModuliParameter$. For completeness (and possible use in other similar questions), we also include this short argument before proving the stronger claim.

\begin{lemma}\label{lemma: no identically torsion section with finitely many exceptions}
 \Cref{prop: no identically torsion section on the whole moduli space} holds with at most finitely many exceptions.
\end{lemma}

\begin{proof} Let $M_{2, 1}$ be the moduli space of curves of genus 2 with a marked point.
Suppose that there exists a family as in the statement. In terms of moduli spaces, the base $\ModuliSpace$ of this family gives a finite degree cover of a hypersurface in $M_{2,1}$ that projects surjectively onto $M_2$ with a map of finite degree.

The construction of \Cref{setup} gives an abelian scheme over a Zariski-open dense subset of a finite cover of $M_{2,1}$, with fibres which are squares of elliptic curves. Moreover, we have a section of this scheme, and its image is not contained in any group subscheme since the components of the section are generically independent by \Cref{rmk: special features of our situation} (B).
The image of $\ModuliSpace$ in $M_{2,1}$ would then correspond to a hypersurface where this section becomes torsion. Theorem 1.1 of \cite{CMZ} asserts the finiteness of such hypersurfaces, which yields the desired conclusion.
\end{proof}

We now prove the stronger claim made in \Cref{prop: no identically torsion section on the whole moduli space}, namely that there are no such sections $q_\ModuliParameter$ at all. 

\begin{proof} [Proof of \Cref{prop: no identically torsion section on the whole moduli space}] 
Replacing $\ModuliSpace$ by a finite cover, we can assume $T'=T$.
It suffices to prove that, for every choice of the section $q_\ModuliParameter$ at infinity on $\tilde X_\ModuliParameter$, the sections $\phi(p_2)-\phi(p_1)$ and $\phi(p_3)-\phi(p_2)$ of the elliptic scheme $E_\ModuliParameter \to \ModuliSpace$ (see \Cref{setup}) are not both identically torsion. Recall from \eqref{eq: affine equation for E} that the elliptic curve $E_\ModuliParameter$ is defined by the affine equation
\begin{equation}\label{E.cubic} 
z^3-3Q_\ModuliParameter(u)z-2P_\ModuliParameter(u)=0,
\end{equation}
where the curve $\tilde X_\ModuliParameter$ of genus $2$ has affine equation $y^2=f_\ModuliParameter(u)=P_\ModuliParameter(u)^2-Q_\ModuliParameter(u)^3$. Here the polynomials $P_\ModuliParameter$ and $Q_\ModuliParameter$ have respective degrees at most $3, 2$; see \Cref{lemma: 3-to-1 parametrisation} for an interpretation of the decomposition of $f_\ModuliParameter$ as $P_\ModuliParameter^2 -Q_\ModuliParameter^3$. 

Replacing $f_\ModuliParameter(u)$ with 
$(cu+d)^6f_\ModuliParameter((au+b)/(cu+d))$, where the coefficients $a, b, c, d$ may depend on $\ModuliParameter$ and $ad\neq bc$, gives an isomorphic curve. One can change $P_\ModuliParameter, Q_\ModuliParameter$ accordingly by applying the same transformation. 
By such a transformation (actually, we only need a translation on $u$) we can and do assume that the $u$-coordinate of $q_\ModuliParameter$ is identically $0$. This is just to simplify the notation and to reduce the number of variables involved. The assumption that $q_\ModuliParameter$ is non-special amounts to $f_\ModuliParameter(0)\neq 0$, but this is not relevant to the present argument.

Recall from \Cref{prop: properties of the construction} (6) that the points $\phi(p_1), \phi(p_2), \phi(p_3)$ (omitting again the dependence on $\ModuliParameter$ for notational convenience) are the three points on $E_\ModuliParameter$ with the same $u$-coordinate as the point at infinity $q_\ModuliParameter$. Thus, with the present normalisation, the points $\phi(p_i)$ have $u=0$ and are defined by $z^3-3Q_\ModuliParameter(0)z-2P_\ModuliParameter(0)=0$, which has no repeated roots if $q_\ModuliParameter$ is non-special.

We now assume by contradiction that the sections $\phi(p_2)-\phi(p_1), \phi(p_3)-\phi(p_2)$ are both identically torsion on $\ModuliSpace$.

We can fibre the threefold $\ModuliSpace$ with the surfaces $\ModuliSpace_c$ defined by the condition that the $j$-invariant of $E_\ModuliParameter$ is equal to $c$ exactly when $\ModuliParameter \in \ModuliSpace_c$. We restrict our attention to one general enough such surface $\ModuliSpace_c$. For $\ModuliParameter \in \ModuliSpace_c$, the cubics defined by \eqref{E.cubic}, or rather the corresponding homogeneous cubics in the variables $(\LastCoord:U:Z)$ (where $u=U/\LastCoord, z=Z/\LastCoord$), all have the same $j$-invariant and are therefore all projectively isomorphic (this is a classical fact; for a reference, see \cite[Corollary 3.5]{MR3777131}).
In particular, any two of the resulting equations in $\LastCoord, U, Z$ are related by a $3\times 3$ invertible linear substitution in the homogeneous variables. 

We now choose a flex on $E_\ModuliParameter$ whose $U$ and $Z$-coordinates are both non-zero (generically on $\ModuliSpace_c$). Note that such a flex exists: there are 9 flexes, and at most 6 have $U=0$ (3 points) or $Z=0$ (another 3 points).
Our chosen flex can be sent to $(0:1:1)$ by a projective transformation which depends algebraically 
on $\ModuliParameter \in \ModuliSpace_c$ and can be represented by a matrix with rows of the form $(*,*,0)$, $(0,*,0)$, $(0,0,*)$. Taking a finite cover of $\ModuliSpace$ over which the flex becomes rational, the corresponding projective transformation can also be chosen to be rational. Moreover, applying such a transformation preserves the property $U(q_\ModuliParameter)=0$. The polynomials $P_\ModuliParameter, Q_\ModuliParameter$ are also transformed according to an element of $\operatorname{PGL}_2$, so the shape of equation \eqref{E.cubic} remains the same. Note that the $\operatorname{PGL}_2$-equivalence class of the pair $(P_\ModuliParameter, Q_\ModuliParameter)$ determines the isomorphism class of $\tilde{X}_\ModuliParameter$, because $\tilde{X}_\ModuliParameter$ can be described by the affine equation $y^2=P_\ModuliParameter(u)^2 - Q_\ModuliParameter(u)^3$. In particular, our projective changes of variables do not alter the isomorphism class of $\tilde{X}_\ModuliParameter$.

Hence, we can assume that for each $\ModuliParameter \in \ModuliSpace_c$ the genus-1 curve $E_\ModuliParameter$ passes through $(0:1:1)$ and has a flex there. We choose this flex as the origin of each of these curves, which then acquire the structure of elliptic curves.

We now fix arbitrarily a point $\ModuliParameter_0 \in \ModuliSpace_c$ and denote by $E^*:=E_{\ModuliParameter_0}$ the corresponding elliptic curve. For every $\ModuliParameter \in \ModuliSpace_c$, there is a projective change of variables $L_\ModuliParameter$ that sends $E_\ModuliParameter$ to $E^*$. By continuity, the transformations $L_\ModuliParameter$ fix the flex $(0:1:1)$: indeed, the flex $(0:1:1)$ of $E_\ModuliParameter$ is sent by $L_\ModuliParameter$ to a flex of $E^*$, and the 9 flexes of $E^*$ form a discrete set. {The algebraic morphism $L_\ModuliParameter : E_\ModuliParameter \to E^*$ therefore sends the origin of the group law of $E_\ModuliParameter$ to that of $E^*$, and is thus a group homomorphism.}

{We claim that the line $U=0$ is sent by $L_\ModuliParameter$ to a line $\ell_0$ which meets $E^*$ in two (and therefore three) torsion points. 
To see this, recall that the points on the curve $E_\ModuliParameter$ given by \eqref{E.cubic} with $U=0$ are precisely the three points $\phi(p_1), \phi(p_2), \phi(p_3)$. Denote by $R_i$ the image of $\phi(p_i)$ via $L_\ModuliParameter$.
By assumption, the sections $\phi(p_1)-\phi(p_2)$ and $\phi(p_3)-\phi(p_2)$ are torsion on $E_\ModuliParameter$. 
Their images $L_\ModuliParameter(\phi(p_1)) - L_\ModuliParameter(\phi(p_2))=R_1-R_2$ and $L_\ModuliParameter(\phi(p_3))-L_\ModuliParameter(\phi(p_2))=R_3-R_2$ are therefore torsion points on $E^*$. Note in particular that the points $R_i$, for $i=1,2,3$, are the three intersections of $\ell_0$ with $E^*$. \\
Given any projective line $\ell \subset \mathbb{P}_2$, the three points of intersection of $\ell$ with $E^*$ (counted with multiplicity) sum to zero, because the chosen origin for the group law on $E^*$ is a flex.
Applying this to the line $\ell_0$, we obtain that in the group $E^*$ the points $R_1-R_2$, $R_3-R_2$ and $R_1+R_2+R_3$ are all torsion, which easily implies that each of $R_1, R_2, R_3$ is torsion (indeed, $3R_2 = (R_1+R_2+R_3)-(R_1-R_2)-(R_3-R_2)$, so $R_2$ is torsion, and therefore so are $R_1=R_2+(R_1-R_2)$ and $R_3=R_2+(R_3-R_2)$). Thus, as $\ModuliParameter$ varies over $\ModuliSpace_c$, the intersection points $\ell_0 \cap E^*$ vary with continuity but are always torsion: this implies that they are actually constant as $\ModuliParameter$ varies, and therefore that $\ell_0$ is constant in $\ModuliParameter$. Since for $\ModuliParameter=\ModuliParameter_0$ we have $\ell_0=\{U=0\}$, we obtain that the transformations $L_\ModuliParameter$ all fix the line $U=0$.
}

By further composing with a linear projective transformation of the form 
\[
\begin{pmatrix}
    1 & 0 & 0 \\
    0 & \rho_\ModuliParameter & 0 \\
    0 & 0 & \rho_\ModuliParameter
\end{pmatrix}
\]
for a suitable $\rho_\ModuliParameter \neq 0$ we can assume, in addition to all the previous properties (shape of the equation \eqref{E.cubic}, choice of flex, isomorphism class of $\tilde{X}_\ModuliParameter$),
that one intersection of $U=0$ with $E_\ModuliParameter$ is constant in $\ModuliParameter$, say equal to $(1:0:1)$. 
(This follows easily from the fact that $z^3-3Q(0)z-2P(0)=0$ has three distinct roots.)

Restricting to a curve $\ModuliSpace'_c$ inside $\ModuliSpace_c$, we can further ensure that a second intersection of $E_\ModuliParameter$ with the line $U=0$ is constant in $\ModuliParameter \in \ModuliSpace'_c$. The form of \eqref{E.cubic} guarantees that the sum of the $z$-coordinates of the three (affine) intersection points of $E_\ModuliParameter$ with $u=0$ vanishes. Therefore, if two intersections are constant, the third one is too.

In summary, 
for all $\ModuliParameter$ in the curve $\ModuliSpace'_c$ the transformation $L_\ModuliParameter^{-1}$ fixes the three points of $E^*$ on $U=0$ and fixes the line $U=0$; this implies that it fixes the line $U=0$ pointwise, as well as fixing the flex $(0:1:1)$.

It follows from these conditions that -- up to a nonzero factor -- for $\ModuliParameter\in \ModuliSpace'_c$ the transformation $L_\ModuliParameter$ is represented by a matrix of the form 
\[
\begin{pmatrix}
1 & 0 & 0 \\
0 & a & 0 \\
0 & a-1 & 1
\end{pmatrix},
\]
where $a$ is a function of $\ModuliParameter$. 
We now recall that $L_\ModuliParameter$ sends the equation \eqref{E.cubic} to the equation of $E^*$, which has the same shape and in particular has a vanishing coefficient of $z^2$. This yields $a=1$, so $L_\ModuliParameter$ is the identity: it follows that all curves $\tilde X_\ModuliParameter$ for $\ModuliParameter \in \ModuliSpace'_c$ are isomorphic. We have reached a contradiction, because by assumption only finitely many $\tilde X_\ModuliParameter$ are isomorphic to a given one, and this proves the assertion.
\end{proof}

\begin{remark}\label{rmk: section is non-constant over any two-dimensional base}
The proof also yields that our section cannot be a (torsion or non-torsion) constant on a subscheme with iso-constant fibres if the base has dimension at least $2$. Note that it makes sense to speak of a `constant' section if the scheme is isotrivial.
\end{remark}

\subsection{Families of genus-2 curves leading to an isotrivial elliptic scheme}

The argument in the proof of \Cref{prop: no identically torsion section on the whole moduli space} allows us to be explicit about the families of genus-2 curves for which our construction yields isotrivial elliptic schemes. Keeping the notation of that proof, we observe that after base-change to $\ModuliSpace_c$ we obtain a trivial square-elliptic family $(E^*)^2\times R_c$, where $R_c$ is a finite cover of $\ModuliSpace_c$, and where our section corresponds to a map $\sigma$ from $R_c$ to $(E^*)^2$. 
We now show that this map is not dominant. We also describe its image, a curve in $(E^*)^2$. 

First, by a translation on the coordinate $u=U/\LastCoord$, we can assume that $u(q_\ModuliParameter)$ is identically zero, and therefore that $\phi(p_{1,\ModuliParameter})$, $\phi(p_{2,\ModuliParameter})$, $\phi(p_{3,\ModuliParameter})$ are the three intersection points of $E_\ModuliParameter$ with the line $U=0$. Recall from \Cref{setup} that the section we are considering is $(\phi(p_{1,\ModuliParameter})-\phi(p_{2,\ModuliParameter}), \phi(p_{2,\ModuliParameter})-\phi(p_{3,\ModuliParameter}))$. 

Fix an elliptic curve $E^*$ in the family and recall that our coordinates on $\mathbb{P}_2$ are denoted by $[\LastCoord: U : Z]$. We can assume that $E^*$ has a flex at $(0:1:1)$, which we take as the origin of the group law, and that it passes through $(1:0:1)$. 
As in the proof of \Cref{prop: no identically torsion section on the whole moduli space}, we can always arrange for $E^*$ to have these properties.

Consider the projective transformations $L_\ModuliParameter$ fixing both the flex $(0:1:1)$ and the point $(1:0:1)$. 
 Their matrices are of the form 
\begin{equation}\label{eq: matrices in the group Gamma}
     \begin{pmatrix}
 a & b & -b \\
 c & d & -c \\
 a-b-e & d-c-e & e
 \end{pmatrix};
 \end{equation}
 these matrices give a 4-dimensional algebraic subgroup $\Gamma$ of $\operatorname{PGL}_3$. This group acts faithfully on the projective plane $\P_2$, and also on the set of equations of the form \eqref{E.cubic} up to scalars (if an element of $\Gamma$ fixes an equation up to scalars, it induces an automorphism of the corresponding elliptic curve with two fixed points. Any such automorphism is either the identity or an involution that acts as $-1$ on the tangent space at the identity. However, by continuity, one sees easily that the action on the tangent space must be the identity). 
 
 \medskip
 
 \begin{definition}\label{def: variety V}
We let $V\subset \Gamma$ be the subvariety consisting of the $L\in \Gamma$ such that the coefficients of $Z^2\LastCoord$ and $Z^2U$ in $L(E^*)$ are $0$.     
 \end{definition}
Note that $V$ has dimension $\ge 2$, because it is the fibre above the point $(0,0)\in\A^2$ of a map $\Gamma\to\A^2$ (it is non-empty because it contains the identity).
 In fact, the dimension of $V$ is exactly $2$. If we had $\dim V \ge 3$, there would exist inside $V$ an irreducible curve corresponding to a single isomorphism class of curves $\tilde X$ (to see this, note that the genus 2 curves parametrised by $R_c$ form a 2-dimensional family in moduli: every curve of genus 2 corresponds to finitely many elliptic curves via our construction, so fixing the $j$-invariant makes the dimension of the family drop by one).

The corresponding polynomials $P^2-Q^3$ would all be related by a homography, up to a constant factor; but then, by continuity, the same would be true individually of the polynomials $P, Q$ (because for any fixed $F(u)$ there are only finitely many solutions $(P, Q)$ to the equation $P(u)^2-Q(u)^3=F(u)$, by\footnote{This may also be proved directly as follows. 
Let $F(u)=P(u)^2-Q(u)^3$, where $F$ is fixed and without multiple roots. If there are infinitely many solutions in polynomials $P,Q$ of degrees $3, 2$, then there is an algebraic family of solutions of dimension $1$. In other words, we may assume that the coefficients of $P,Q$ are rational functions on a certain curve $Z$, not all constant. Let $\delta$ be a derivation on $\C(Z)$ with constant field $\C$. Applying $\delta$ to the equation we obtain $2(\delta P)P=3(\delta Q)Q^2$. On the other hand, $P$ and $Q$ must be coprime, hence $Q^2$ divides $\delta P$ and $P$ divides $\delta Q$, forcing $\delta P,\delta Q$ to vanish. This implies that $P, Q$ are constant, contradiction.} Lemma \ref{lemma: 3-torsion points}.)
 In conclusion, we would have a $1$-parameter family of matrices in $\Gamma$ carrying a fixed cubic $z^3-3Q(u)z-2P(u)$ into a multiple of $z^3-3\rho^2(cu+d)^2Q(\gamma(u))z-2\rho^3(cu+d)^3P(\gamma(u))$ for a nonzero $\rho$ and a M\"obius transformation $\gamma(u)=(au+b)/(cu+d)$. This would give a $1$-dimensional family of automorphisms $(u,z)\mapsto (\gamma(u),z/\rho(cu+d))$ of the corresponding elliptic curve, which is impossible (e.g., by the general theory; alternatively, note that these transformations would have to stabilise the set of flexes, or else the branch points of the map $u:E^*\to\P_1$). 

 \begin{remark}
      Naturally, it is possible to prove the contention about $\dim V$ also by explicit calculation. We could also use \Cref{prop: no identically torsion section on the whole moduli space}: if $V$ had dimension $3$, we would get a subvariety of dimension $2$ over which our section is constant, which contradicts \Cref{rmk: section is non-constant over any two-dimensional base}.
  \end{remark}

Note that -- with our normalisations -- the family $E_\ModuliParameter$ is contained in the family of plane cubics that are projectively equivalent to $E^*$, pass through $(1 : 0: 1)$, and have a flex at $(0 : 1 : 1)$. By definition, the family of all such plane cubics is given by $\{ \gamma(E^*) : \gamma \in \Gamma\}$. Moreover, every curve in the family $E_\ModuliParameter$ is represented by any equation for which the coefficients of $Z^2\LastCoord$ and $Z^2U$ vanish, hence the 2-dimensional family $E_\ModuliParameter$ is contained in the $2$-dimensional family $\{ L(E^*) : L \in V\}$.
Thus, the cubic curves $E_\ModuliParameter$ may be taken to be the curves $E_L:=L(E^*)$ for $L\in V$: they are all isomorphic, but they come from generically non-isomorphic curves $\tilde X$.

Denote by $L_\ModuliParameter$ the projective transformation in $V$ such that $L_\ModuliParameter(E^*)=E_\ModuliParameter$. Trivially, $L_\ModuliParameter^{-1}$ gives an isomorphism between $L_\ModuliParameter(E^*)$ and $E^*$: this isomorphism takes the line $U=0$ to a line through $(1:0:1)$, hence sends the three points $\phi(p_{1,\ModuliParameter}), \phi(p_{2,\ModuliParameter}), \phi(p_{3,\ModuliParameter})$ to the three intersection points of $L_\ModuliParameter^{-1}(\{U=0\})$ with $E^*$. Thus, seen as a map $R_c \to (E^*)^2$, the section $\sigma$ takes the form
\[
s \mapsto (A_\ModuliParameter - B_\ModuliParameter, B_\ModuliParameter-C_\ModuliParameter),
\]
where $A_\ModuliParameter = (1:0:1), B_\ModuliParameter, C_\ModuliParameter$ are the three intersection points of $L_\ModuliParameter^{-1}(\{U=0\})$ with $E^*$. Since $A_\ModuliParameter+B_\ModuliParameter+C_\ModuliParameter=0$ in $E^*$ (they lie on the same line and the origin is a flex), we can rewrite the above as
\[
s \mapsto (A_\ModuliParameter - B_\ModuliParameter, A_\ModuliParameter+2B_\ModuliParameter),
\]
which shows that the image of the section is contained in the 1-dimensional subvariety 
\begin{equation}\label{eq: image of section for constant j-invariant}
\{ (x,y) : 2x+y = 3(1:0:1) \}
\end{equation}
of $(E^*)^2$.

Note that the isomorphism class of $E^*$ is determined by the chosen parameter $c$, whereas the point with coordinates $(1:0:1)$ in our model is determined by the choice of section at infinity.

\begin{remark}
    Continuing with the notation above, 
    we note that when the point $(1:0:1)$ is not torsion on $E^*$ our method does not yield anything on the curves parametrised by the space $\ModuliSpace_c$. On the contrary, when it is torsion, we have pencils like in the example presented in \S\ref{subsect: pencil}.
\end{remark}

We further clarify the geometry of families of curves of genus $2$ for which our doubly-elliptic scheme is isotrivial. We have shown that any such family is given (up to finite cover) by the images of a fixed genus-2 curve through linear substitutions parametrised by the variety $V$ of \Cref{def: variety V}. We now show that the geometry of $V$ is as simple as possible.

\begin{proposition}\label{prop: V is rational}
    The variety $V$ appearing in \Cref{def: variety V} is rational (over $\overline{\Q}$).
\end{proposition}
\begin{proof}
    Replace $V$ with its projective closure $\overline{V}$ in $\mathbb{P}(a,b,c,d,e) \cong \mathbb{P}_4$. It suffices to show that $\overline{V}$ is rational. One can check easily that $\overline{V}$ contains the line $L \cong \mathbb{P}_1$ given by $b=c=e=0$. Note that $\overline{V}$ is by definition the intersection of two cubics $X_{UZ^2}(a,b,c,d,e)=0$ and $ X_{\LastCoord Z^2}(a,b,c,d,e)=0$ in $\mathbb{P}_4$ (we impose two conditions, the vanishing of the coefficients $UZ^2, \LastCoord Z^2$, and each such coefficient is a homogeneous polynomial of degree 3 in the variables $a, b, c, d, e$ of \eqref{eq: matrices in the group Gamma}).

    Consider now any hyperplane $H$ in $\mathbb{P}_4$ containing $L$ (these hyperplanes $H$ can be rationally parametrised). The intersection $C_H := H \cap \overline{V}$ is the intersection of two cubics in $H \cong \mathbb{P}_3$, hence it is a curve of degree 9.
    
We claim that $C_H$ contains $L$ with multiplicity $4$, hence there is a residual curve $R_H$ of degree 5. Moreover, the points $s_1 := [0:0:0:1:0]$ and $s_2 := [1:0:0:0:0]$ are singular points of $R_H$.

Assume for the moment that $\overline{V} \cap P$ contains $L$ with multiplicity 4 (we discuss below how to prove this): we then show how to check that $s_1$ is a singular point of $R_H$. The proof that $s_2$ is a singular point of $R_H$ and that $\overline{V} \cap P$ does, in fact, contain $L$ with multiplicity 4 is similar.

Write the polynomials $P$ and $Q$ as
\[
P = A_0 U^3 + A_1 U^2 \LastCoord + A_2 U \LastCoord^2 + A_3 \LastCoord^3, \quad Q = B_0 U^2 + B_1  U  \LastCoord + B_2 \LastCoord^2.
\]
We work in the affine chart $d \neq 0$, which contains the point $s_1$. We thus set $d=1$; note that our candidate singular point is the origin of the resulting affine space $\mathbb{A}^4$. We then compute two explicit polynomials generating the ideal defining $\overline{V}$. The initial forms of these polynomials at the origin are the quadric
\[
(-2A_2 - 3B_2)b^2 + (-4A_1 - 3B_1)bc + 3B_1be + (-6A_0 - 3B_0)c^2 + 6B_0ce + 3e^2
\]
and the cubic
\[
\begin{aligned}
(-6A_3 - 3B_2)ab^2 & + (-4A_2 - 3B_1)abc + 6B_2abe + (-2A_1 - 3B_0)ac^2 + 3B_1ace
    \\ & + 3ae^2 + 3B_2b^3 + (-2A_2 + 3B_1)b^2c + 3B_2b^2e + (-4A_1 + 3B_0)bc^2 \\
    & +    6B_1bce - 3be^2 - 6A_0c^3 + 9B_0c^2e - 3e^3.
    \end{aligned}
\]
Thus, the vertex of the tangent cone to $\overline{V}$ at $s_1$ is a point of multiplicity 6. Taking away multiplicity 4 (coming from the line $L$), we see that $s_1$ is a point of multiplicity 2 of the residual curve $R_H$, hence a singular point. 

To check the statement that the line $L$ is contained in $H \cap \overline{V}$ with multiplicity 4, one similarly computes the tangent cone at any point with $b=c=e=0$: it turns out that this cone is defined by two initial forms of degree (at least) 2, hence the point has multiplicity (at least) 4 and $\overline{V} \cap H$ contains $L$ with multiplicity 4.

Moreover, we claim that $R_H$ is (generically) not contained in a hyperplane of $H$. Indeed, suppose that $R_H$ is contained in a hyperplane $H' \cong \mathbb{P}_2$ of $H$. The curve $R_H$ would in particular be contained in $H' \cap \{X_{\LastCoord Z^2}=0\}$, where $X_{\LastCoord Z^2}$ is the cubic form discussed above. Unless $X_{\LastCoord Z^2}=0$ contains $H'$ (which does not happen for a general $H$), the intersection is a curve of degree $3 = \deg X_{\LastCoord Z^2}$, hence it cannot contain the degree-5 curve $R_H$.

To summarise, we now know that for a generic hyperplane $H \cong \mathbb{P}_3$ containing $L$ the intersection $\overline{V} \cap H$ consists of the line $L$ with multiplicity $4$ together with a curve $R_H$ of degree 5, possessing two singular points $s_1 = [0:0:0:1:0]$ and $s_2 := [1:0:0:0:0]$.

Thus, projecting $R_H$ from $s_1$ to $\mathbb{P}_2$ gives a birational morphism from $R_H$ to a curve of degree 3 in $\mathbb{P}_2$ with a singular point (the image of $s_2$). A plane cubic with a singular point is rational, hence $R_H$ is also rational. The conclusion is that $\overline{V}$ can be covered by a rationally parametrised family of rational curves,
and is therefore rational over $\overline{\Q}$ (indeed, we have shown that $V$ is unirational, and Castelnuovo's theorem shows that unirational surfaces over $\overline{\Q}$ are rational). 
\end{proof}

\begin{remark}\label{rmk: moving the point at infinity}
    One can ask a complementary question: fix a smooth projective curve $X$ of genus $2$ and consider all possible choices of point at infinity $q \in X(\overline{\Q})$. The construction of \Cref{setup} gives a constant abelian scheme $E^2$ over $X$, together with a section $\sigma = (\phi(p_1)-\phi(p_2), \phi(p_2)-\phi(p_3))$ of $E^2$ (the section depends on $q \in X$). We show that this section cannot be identically torsion. To see this, fix an origin $O$ on $E$ that is a flex. 
    By \Cref{setup} (6), the points $\phi(p_i)$ are the points on $E$ (as given by \eqref{eq: affine equation for E}) with $x$-coordinate equal to $x(q)$. Since we have put the origin of the group law at a flex and the points $\phi(p_i)$ lie on the same line, we have $\sum_i \phi(p_i)=0$. Hence, if the differences $\phi(p_1)-\phi(p_2), \phi(p_2)-\phi(p_3)$ are torsion, so is each of the points $\phi(p_1), \phi(p_2), \phi(p_3)$. In particular, if $\sigma$ is identically torsion, then so is $\phi(p_i)$ for every $i$.
    As $x(q)$ varies, the points $\phi(p_i)$ are given by the intersection of \eqref{eq: affine equation for E} with any line $x=\text{constant}$. By the above argument, if $\sigma$ is identically torsion, any such $\phi(p_i)$ is torsion, contradiction (we would have that every point in $E(\overline{\Q})$ is torsion).
We note that \cite[Theorem 1.3]{LP} gives the existence of infinitely many choices of $q \in X(\mathbb{C})$ for which \textit{no} connected finite étale cover admits non-trivial morphisms to $\mathbb{G}_m$. Interesting as this is, it is hard to see whether the same would hold over $\overline{\Q}$.
\end{remark}

\section{Two-parameter families of genus-2 curves}\label{sect: moduli space of genus 2 curves with marked 3-torsion points}

In this section we consider a partial extension of \Cref{T} to the case of families of curves $X_\ModuliParameter \to \ModuliSpace$ of genus 2 over bases $\ModuliSpace$ of dimension 2, see Theorems \ref{thm: dichotomy for families} and \ref{thm: two-parameters family body} below. If we try to mimic the proof of \Cref{T}, a difficulty that we face is that it is in principle possible for the abelian scheme $E_\ModuliParameter^2 \to \ModuliSpace$ (constructed from a choice of three-torsion section of $\Jac(\tilde{X})$) to be isotrivial. Our first objective is to show that this cannot happen for all possible choices of three-torsion sections. We thus need to understand the interactions between different choices of three-torsion sections.

To this end, we begin by recalling a description, due to Bruin-Flynn-Testa \cite{MR3263947}, of (a Zariski-open subvariety of) the moduli space $\mathcal{A}_2(\Sigma)$ of triples $(C, R_1, R_2)$, where $C$ is a smooth projective curve of genus $2$ and $R_1, R_2$ are three-torsion points of the Jacobian $J_C$ such that $e_3(R_1, R_2)=1$, where $e_3 : J_C[3] \times J_C[3] \to \mu_3$ is the Weil pairing on the 3-torsion. 
From this, we will deduce a description of (an open subset of) the moduli space of quadruples $(C, (P_1, Q_1), (P_2, Q_2), (P_3, Q_3))$, where:
\begin{enumerate}
    \item $C$ is a genus-2 curve, represented by the hyperelliptic model $y^2=f(x)$ with $f(x)$ separable of degree 6;
    \item $(P_i, Q_i)$ are polynomials of respective degrees $3, 2$ such that $f(x) = P_i(x)^2-Q_i(x)^3$; in particular, $(P_i, Q_i)$ corresponds to a 3-torsion point $R_i$ on $\Jac(C)$ (see Lemma \ref{lemma: 3-to-1 parametrisation});
    \item we have $e_3(R_1, R_2)=1$ and $R_3=R_1+R_2$ in $\Jac(C)[3]$.
\end{enumerate}

\begin{theorem}[{\cite[Theorem 6]{MR3263947}}]
The following hold:
\begin{enumerate}
    \item For $i=1, \ldots, 4$ there are polynomials $G_i, H_i \in \mathbb{Q}(r, s, t)[x]$, with $H_i(x)$ monic, and elements $\lambda_i \in \mathbb{Q}(r, s, t)^\times$ such that $f_{r, s, t} := G_i^2 + \lambda_i H_i^3$ is independent of $i$ and is a separable polynomial of degree 6 with coefficients in $\mathbb{Q}(r, s, t)$. One can take
    \[
    G_1 =  (s - st - 1)x^3 + 3s(r - t)x^2 + 3sr(r - t)x - st^2 + sr^3 + t, H_1 = x^2 + rx + t, \lambda_1 = 4s;
    \]
    \[
    G_2 = (s - st + 1)x^3 + 3s(r - t)x^2 + 3sr(r - t)x - st^2 + sr^3 - t, H_2 = x^2 + x + r, \lambda_2 = 4st;
    \]
    \[
    \begin{aligned}
            G_3=
\frac{1}{st+1}\big( & (s^2t^2 - s^2t + 2st + s + 1)x^3 + (3s^2t^2 - 3s^2tr + 3st + 3sr)x^2 \\
& + (3s^2t^2r - 3s^2tr^2 + 3str + 3sr^2)x + s^2t^3 - s^2tr^3 + 2st^2 + sr^3 + t\big),
    \end{aligned}
    \]
    \[
    H_3= sx^2 + (2sr - st - 1)x + sr^2, \lambda_3 = \frac{4t}{(st+1)^2}.
    \]
    \item Denote by $\mathcal{C}_{r, s, t}$ the curve over $\mathbb{Q}(r, s, t)$ given by $y^2 = f_{r, s, t}(x)$ and by $R_i$ the 3-torsion point on $\Jac(\mathcal{C}_{r, s, t})$ corresponding to the decomposition $f_{r, s, t} = G_i^2 + \lambda H_i^3$ (see \Cref{lemma: 3-torsion points}). The subgroup of $\Jac(\mathcal{C}_{r, s, t})$ generated by $R_1, R_2$ has order 9 and consists of the identity, together with the eight 3-torsion points $\pm R_1, \pm R_2, \pm R_3, \pm R_4$ (these eight points are non-trivial and distinct). Moreover, $e_3(R_1, R_2)=1$ and $R_3=R_1+R_2$.
    \item Let $K$ be a number field. Any sufficiently general triple $(C, R_1, R_2)$, where $C/K$ is a curve of genus 2, $R_1, R_2 \in \Jac(C)(K)[3]$ generate a subgroup of order $9$, and $e_3(R_1, R_2)=1$, is a specialisation of $\mathcal{C}_{r, s, t}$ at suitable $(r, s, t) \in K^3$.
\end{enumerate}
    
\end{theorem}

\begin{remark}
    The result stated in \cite{MR3263947} also gives explicit expressions for $G_4, H_4, \lambda_4$ and holds more generally over fields of characteristic at least 5. We will not need these facts.
\end{remark}

In other words, the curve $\mathcal{C}_{r, s, t}$ is the universal curve over an open subscheme of the modular variety $\mathcal{A}_2(\Sigma)$. Note that this result shows in particular that $\mathcal{A}_2(\Sigma)$ is rational. 
We are interested in $\mathcal{A}_2(\Sigma)$ because it is a natural moduli space over which we can carry out our construction of the doubly-elliptic scheme $E^2$ corresponding to a choice of point of order $3$. In fact, given our description \eqref{eq: affine equation for E} of the elliptic curve $E$, we also need a trivialisation of each degree-3 étale cover that we want to consider. Concretely, we want to write $f_{r, s, t}(x)$ in the form $P(x)^2-Q(x)^3$ of \Cref{lemma: 3-to-1 parametrisation}. In order to be able to do so rationally, we pass to the cover $\mathcal{A}$ of $\mathcal{A}_2(\Sigma)$ given by the equations
\[
\begin{cases}
    s = -\frac{1}{4}s_1^3 \\
    t = t_1^3 \\
    -\frac{4t}{(st+1)^2} = u_1^3.
\end{cases}
\]
On this cover (which, unfortunately, is no longer rational) we can write
\[
f_{r, s, t}(x) = G_1^2 - (s_1 H_1)^3 = G_2^2 - (s_1t_1H_2)^3 = G_3^2 - (u_1H_3)^3.
\]
Thus, we have an explicit moduli space of genus-2 curves whose Jacobians are equipped with three distinct 3-torsion points, namely those given by $(P_1, Q_1) = (G_1, s_1H_1)$, $(P_2, Q_2)=(G_2, s_1t_1H_2)$ and $(P_3, Q_3)=(G_3, u_1H_3)$ under the parametrisation of \Cref{lemma: 3-to-1 parametrisation}. For each point $a \in \mathcal{A}$ we then have the three elliptic curves $E_1= E_{1, a}, E_2=E_{2, a}$ and $E_3=E_{3,a}$ given by \eqref{eq: affine equation for E} for the pair $(P_i, Q_i)$. In particular, we have three well-defined $j$-invariant functions $j_i(a) = j(E_{i,a})$ for $i=1, 2, 3$.
\begin{proposition}\label{prop: joint j-invariants}
    The map
    \[
    \begin{array}{cccc}
    J : & \mathcal{A} &\to & \mathbb{A}^3 \\
    & a & \mapsto & (j_1(a), j_2(a), j_3(a))
    \end{array}
    \]
    is dominant.
\end{proposition}
\begin{proof}
    Both the source and the target of $J$ are smooth varieties of dimension 3, so it suffices to check that $J$ is generically étale. Note that the elliptic curves $E_1, E_2, E_3$ are actually defined as soon as the torsion points $R_1, R_2, R_3$ are, so their $j$-invariants are well-defined functions on $\mathcal{A}_2(\Sigma)$ (it is only our model \eqref{eq: affine equation for E} that only makes sense on $\mathcal{A}$). Thus, $J$ factors via $\mathcal{A}_2(\Sigma)$, and it suffices to check that $J : \mathcal{A}_2(\Sigma) \to \mathbb{A}^3$ is generically étale.
    By the Jacobian criterion, it suffices to check that the determinant of the matrix
    \[
    \begin{pmatrix}
        \frac{\partial j_1}{\partial r} & \frac{\partial j_1}{\partial s} & \frac{\partial j_1}{\partial t} \\
        \frac{\partial j_2}{\partial r} & \frac{\partial j_2}{\partial s} & \frac{\partial j_2}{\partial t} \\
        \frac{\partial j_3}{\partial r} & \frac{\partial j_3}{\partial s} & \frac{\partial j_3}{\partial t}
    \end{pmatrix}
    \]
    is non-zero as an element of $\mathbb{Q}(r, s, t)$, and since we have explicit expressions for all our objects, this is a direct calculation that is not hard to carry out on a computer \cite{Computations}.
\end{proof}

\subsection{Two-dimensional families}

Consider now the following situation: we let $\tilde{X}_\ModuliParameter$ be a family of (smooth projective) curves of genus $2$ over a base $\ModuliSpace$ with the property that the image of the moduli map $\ModuliSpace \to M_2$ has dimension 2. Suppose given a section $q_t$ of $\ModuliSpace \to M_2$, serving as the choice of point at infinity on $\tilde{X}_{\ModuliParameter}$. By passing to a suitable cover of $\ModuliSpace$ as usual, and replacing it with a Zariski-open subset if necessary, we can assume that the following data are all defined over $\ModuliSpace$:
\begin{enumerate}
    \item a Weierstrass equation $y^2=f_\ModuliParameter(x)$ for $\tilde{X}_\ModuliParameter$;
    \item three decompositions $f_\ModuliParameter(x) = P_{1, \ModuliParameter}(x)^2 - Q_{1,\ModuliParameter}(x)^3= P_{2, \ModuliParameter}(x)^2 - Q_{2,\ModuliParameter}(x)^3= P_{3, \ModuliParameter}(x)^2 - Q_{3,\ModuliParameter}(x)^3$ such that the 3-torsion points $R_1, R_2, R_3$ corresponding to these decompositions (\Cref{lemma: 3-to-1 parametrisation}) satisfy $\langle R_1, R_2 \rangle \cong (\mathbb{Z}/3\mathbb{Z})^2$ and $e_1(R_1, R_2)=1$;
    \item for each $i=1,2,3$, a section $O_{i,\ModuliParameter}$ of the relative genus 1 curve $E_{i,\ModuliParameter} \to \ModuliSpace$ given by \eqref{eq: affine equation for E} such that $O_{i, \ModuliParameter}$ is a flex of $E_{i, \ModuliParameter}$ for every $\ModuliParameter \in \ModuliSpace$ (using $O_{i, \ModuliParameter}$ as neutral point makes $E_{i, \ModuliParameter}$ into an elliptic curve);
    \item for each $i=1,2,3$, the sections $\phi_i(p_{1,\ModuliParameter})$, $\phi_i(p_{2,\ModuliParameter})$, $\phi_i(p_{3,\ModuliParameter})$ of $E_{i,\ModuliParameter}$ of \Cref{setup}.
\end{enumerate}

\Cref{prop: joint j-invariants} implies that there exists $i \in \{1,2,3\}$ such that $E_{i,\ModuliParameter} \to \ModuliSpace$ is non-isotrivial. Without loss of generality, assume that $i=1$ and write $E_\ModuliParameter = E_{i,\ModuliParameter}$.
Denote by $\sigma = (\sigma_1, \sigma_2)$ the section 
\begin{equation}\label{eq: section}
    \sigma(\ModuliParameter) = 
(\phi(p_{1,\ModuliParameter})-\phi(p_{2,\ModuliParameter}), \phi(p_{2,\ModuliParameter})-\phi(p_{3,\ModuliParameter}))
\end{equation}
of the abelian scheme 
\begin{equation}\label{eq: abelian scheme}
    \mathcal{A} = E_\ModuliParameter^2.
\end{equation} 
For every point $\ModuliParameter \in \ModuliSpace$ such that the differences $\phi(p_{1,\ModuliParameter})-\phi(p_{2,\ModuliParameter})$ and $\phi(p_{2,\ModuliParameter})-\phi(p_{3,\ModuliParameter})$ are both torsion on $E_\ModuliParameter$, Bilu's criterion (\Cref{thm: Bilu version 3}) applies and shows that there is effectivity for the integral points on $\tilde{X}_\ModuliParameter$. 

To understand whether $\sigma$ takes infinitely many torsion values, we apply the theory of the \textit{Betti map}, introduced in \cite{MR2918151} and further developed in \cite{ACZ, MR4202453}. Recall that the Betti map associated with a section $\sigma$ of an abelian scheme $\pi_\ModuliSpace : \mathcal{A} \to \ModuliSpace$ (of relative dimension $g$) may be defined as follows. Fix a point $\ModuliParameter_0 \in \ModuliSpace(\C)$ and a sufficiently small, simply connected neighbourhood $\Delta$ of $\ModuliParameter_0$ in $\ModuliSpace(\C)$. We may define holomorphic functions $\omega_1(\ModuliParameter),\ldots,\omega_{2g}(\ModuliParameter)$ on $\Delta$ that give a basis of the periods of the abelian variety $\mathcal{A}_\ModuliParameter$ for $\ModuliParameter \in \Delta$. The Betti map is the function
\[
b_\Delta : \pi_\ModuliSpace^{-1}(\Delta) \to (\mathbb{R}/\mathbb{Z})^{2g}
\]
defined as follows. Given a point $x \in \mathcal{A}_\ModuliParameter$, we can write $x = \sum_{i=1}^{2g} b_i(x) \omega_i(\ModuliParameter)$, where $b_i(x) \in \mathbb{R}$ is well-defined up to translation by integers. We set $b_\Delta(x) = \left( b_1(x), \ldots, b_{2g}(x) \right)$. A point $x \in \pi^{-1}(\Delta)$ is torsion if and only if $b_\Delta(x)$ is in $(\mathbb{Q}/\mathbb{Z})^{2g}$. Since we are interested in torsion values of the section $\sigma$, we will restrict $b_\Delta$ to the image of $\sigma$ (that is, to the surface $\sigma(\ModuliSpace)$).

By \cite[Proposition 2.1.1]{ACZ}, a sufficient condition for $\sigma_\ModuliParameter$ to be torsion for a complex-analytically dense set of points $\ModuliParameter \in \ModuliSpace(\C)$ is that the generic rank of $b_\Delta$ (for any neighbourhood $\Delta$) is equal to $2g=4$. Note that -- since all the data in our problem are defined over a number field -- the points $\ModuliParameter \in \ModuliSpace(\C)$ for which $\sigma_\ModuliParameter$ is torsion all lie in $\ModuliSpace(\overline{\Q})$.
Consider the Cartesian diagram
\begin{equation}\label{eq: pullback from universal families}
\xymatrix{
\mathcal{A} \ar[r]^{\iota} \ar[d]_{\pi_\ModuliSpace} \pullbackcorner & \mathfrak{A}_2 \ar[d]^{\pi} \\
\ModuliSpace \ar[r]^{\iota_{\ModuliSpace}} & \A_2,
}
\end{equation}
where $\mathbb{A}_2$ is the moduli space of principally polarised abelian surfaces, $\mathfrak{A}_2$ is the universal abelian variety over it, and $\iota, \iota_\ModuliSpace$ are the moduli maps induced by $\pi_\ModuliSpace$. Denote by $W$ the image $\sigma(\ModuliSpace)$ of our section. Note that the image of $\iota_\ModuliSpace$ is $1$-dimensional, because our abelian scheme is the square of an elliptic scheme, and therefore $\iota_\ModuliSpace$ factors via the 1-dimensional moduli space $\mathbb{A}_1$. We may then apply \cite[(1.4)]{MR4202453} to study the generic rank of the Betti map. To state this criterion we need one more piece of notation \cite[Definition A.1]{MR4202453}:
\begin{definition}\label{def: gen-sp}
Denote by $\langle W \rangle_{\textrm{gen-sp}}$ the smallest subvariety of $\mathcal{A}$ which contains $W$ and is of the following form:
up to taking a finite cover of $\ModuliSpace$, we have $\langle W \rangle_{\textrm{gen-sp}} = \tau + \mathcal{Z} + \mathcal{B}$, where $\mathcal{B}$ is an abelian subscheme of $\mathcal{A} \rightarrow \ModuliSpace$, $\tau$ is a torsion section of $\mathcal{A} \rightarrow \ModuliSpace$, and $\mathcal{Z} = Z \times \ModuliSpace$, where $C \times \ModuliSpace$ is the largest constant abelian subscheme of $\mathcal{A} \rightarrow \ModuliSpace$ and $Z \subseteq C$.
\end{definition}

The criterion given in \cite[(1.4)]{MR4202453} shows that in our situation (namely, $\dim \iota_\ModuliSpace(\ModuliSpace)=1$) the generic rank of the Betti map is
\begin{equation}\label{eq: Gao}
    2 \min\left( \dim \iota(W), \dim \langle W \rangle_{\textrm{gen-sp}} - \dim(\ModuliSpace) \right).
\end{equation}

Since the $j$-invariant of $E_\ModuliParameter$ is non-constant, $\mathcal{A} \cong E_\ModuliParameter^2$ admits no constant abelian subscheme of positive dimension, so the scheme $\mathcal{Z}$ in Definition \ref{def: gen-sp} is $\{0\}$. We can then distinguish three cases, according to the relative dimension of the scheme $\mathcal{B}$ appearing in Definition \ref{def: gen-sp}:
\begin{enumerate}
    \item Relative dimension $0$: by definition, this means that $\sigma$ is torsion, hence it certainly takes torsion values at a dense subset of $\ModuliSpace(\C)$.
    \item Relative dimension $1$: this implies that (up to torsion) $\sigma$ actually takes values in an elliptic scheme. Since the elliptic scheme in question is non-constant, it is well-known that the set of points where $\sigma$ is torsion is dense in the base.
    \item Relative dimension $2$: in this case, $\dim \langle W \rangle_{\textrm{gen-sp}}=4$, and Gao's formula \eqref{eq: Gao} for the generic rank of the Betti map shows that this generic rank is equal to $2 \dim \iota(W)$. Since $W$ is a surface (it is isomorphic to $\ModuliSpace$) and $\iota(W)$ dominates the curve $\iota_\ModuliSpace(\ModuliSpace)$, there are only two possibilities for $\dim \iota(W)$: either $\dim \iota(W)=2$ or $\dim \iota(W)=1$. In the former case, the generic rank of the Betti map is $2 \dim \iota(W) = 4=2g$, and as already recalled, \cite[Proposition 2.1.1]{ACZ} shows that the set of points $\ModuliParameter \in \ModuliSpace(\C)$ for which $\sigma_\ModuliParameter$ is torsion is complex-analytically dense.
\end{enumerate}
Thus, we are left with understanding the case $\dim \langle W \rangle_{\textrm{gen-sp}}=4$ and $\dim \iota(W)=1$. Note that in this case $\iota(W)$ is finite over $\iota_\ModuliSpace(\ModuliSpace)$. We reformulate this condition in more concrete terms. Fix a $j$-invariant $j_0$ and consider the locus $\ModuliSpace_{j_0}$ of points in $\ModuliSpace$ where $j(E_{\ModuliParameter})=j_0$: since $j$ is non-constant, $\ModuliSpace_{j_0}$ has dimension $1$. The image $\iota_\ModuliSpace(\ModuliSpace_{j_0})$ in $\mathbb{A}_2$ is a point, hence $\iota(\sigma(\ModuliSpace_{j_0}))$ is zero-dimensional. This means that $\sigma$ takes only finitely many values on $\ModuliSpace_{j_0}$, and in particular, it's constant on each connected component. We have proved the following result:
\begin{theorem}\label{thm: dichotomy for families}
    Let $\tilde{X}_\ModuliParameter$ be a family of (smooth projective) curves of genus $2$ over a base $\ModuliSpace$ with the property that the image of the corresponding moduli map $\ModuliSpace \to M_2$ has dimension 2. Suppose given a section $q_t$ of $\tilde{X}_{\ModuliParameter} \to \ModuliSpace$. Define an abelian scheme $\mathcal{A}_\ModuliParameter \cong E_\ModuliParameter^2 \to \ModuliSpace$ as in \eqref{eq: abelian scheme} and a section $\sigma$ of $\mathcal{A}_{\ModuliParameter}$ as in \eqref{eq: section}. One of the following holds:
    \begin{enumerate}
        \item there is a complex-analytically dense set of points $\ModuliParameter \in \ModuliSpace(\overline{\Q})$ for which there is effectivity for the integral points of $\tilde{X}_\ModuliParameter$;
        \item for each $j_0 \in \C$, the section $\sigma_\ModuliParameter$ is constant on each connected component of the subvariety $\{\ModuliParameter \in \ModuliSpace : j(E_\ModuliParameter)=j_0\}$.
    \end{enumerate}
\end{theorem}

\begin{remark}
In the general setting of abelian schemes, 
it is easy to construct square-elliptic schemes $E_\ModuliParameter^2 \to \ModuliSpace$ with $\dim(\ModuliSpace)=2$ and sections $\sigma$ of $E_\ModuliParameter^2$ such that, on each curve $\ModuliSpace_{j_0}=\{\ModuliParameter \in \ModuliSpace : j(E_\ModuliParameter) = j_0\}$, the section $\sigma$ is constant. There is then no guarantee that the set of $\ModuliParameter \in \ModuliSpace(\C)$ such that $\sigma_\ModuliParameter$ is torsion is dense in $\ModuliSpace$.

\end{remark}

\begin{question}\label{question: method works on every surface}
For a given family $\tilde{X}_\ModuliParameter$, by Proposition \ref{prop: joint j-invariants} there are at least two independent square-elliptic schemes that one can consider, and it seems highly unlikely that case (2) in \Cref{thm: dichotomy for families} happens for \textit{both} of these abelian schemes. It would be interesting to prove that this is in fact impossible.
\end{question}

We now give a concrete application of \Cref{thm: dichotomy for families} by proving \Cref{thm: two-parameters family intro}, which we reproduce here for the reader's convenience:
\begin{theorem}\label{thm: two-parameters family body}
    Let $q_{a,b}$ be the unique point at infinity in the projective completion $\tilde{X}_{a,b}$ of 
    \[
    X_{a,b} : y^4+ay^2-xy-x^3+bx^2=0.
    \]
    There is a complex-analytically dense set of algebraic points $(a,b) \in \overline{\Q}^2$ such that the integral points on $X_{a,b} \colonequals \tilde{X}_{a,b} \setminus \{q_{a,b}\}$ can be determined effectively, over every number field.
\end{theorem}
\begin{proof}
We consider $X_{a,b}$ as a family over an open subscheme $\ModuliSpace$ of $\mathbb{A}^2$.
    It is easy to see that the image of the moduli map $\ModuliSpace \to \mathbb{A}_2$ is 2-dimensional.
    By \Cref{thm: dichotomy for families}, it suffices to check that the section \eqref{eq: section} of the abelian scheme $\mathcal{A}_\ModuliParameter = E_\ModuliParameter^2 \to \ModuliSpace$ (see \eqref{eq: abelian scheme}) is not constant when restricted to the irreducible components of the curves in $\ModuliSpace$ with constant $j$-invariant. We prove this computationally, using a specialisation argument. More precisely: for every $j$-invariant $j_0$, let $\ModuliSpace_{j_0} = \{\ModuliParameter \in \ModuliSpace : j(E_\ModuliParameter)=j_0\}$. We have to show that (for at least one value of $j_0$) the restriction of $\sigma_\ModuliParameter$ to $\ModuliSpace_{j_0}$ is not constant on irreducible components. Suppose by contradiction that $\sigma_\ModuliParameter$ is constant when restricted to each irreducible component of $\ModuliSpace_{j_0}$: then the same would hold after reduction modulo $p$, and in particular, $\sigma_\ModuliParameter$ would take at most as many values as there are geometrically irreducible components in $\ModuliSpace_{j_0}$. Unfortunately, the decomposition into irreducible components of $\ModuliSpace_{j_0}$ is computationally unfeasible: in order to even define the curve $E_\ModuliParameter$ (hence the $j$-invariant $j(E_\ModuliParameter)$ and the variety $\ModuliSpace_{j_0}$), we need to replace $\ModuliSpace$ with a cover $\ModuliSpace'$ of degree 80 (over which a 3-torsion point is defined), followed by a cover $\ModuliSpace''$ of degree 2 (to use the parametrisation of 3-torsion points given in Lemma \ref{lemma: 3-to-1 parametrisation}). Furthermore, the equations defining the subvarieties of $\ModuliSpace''$ with constant $j(E_\ModuliParameter)$ are extremely unwieldy. We then proceed in a much more direct way: we fix $j_0=1$ and compute the degree of the projective closure of the scheme $\ModuliSpace''_0 \colonequals \{\ModuliParameter \in \ModuliSpace'' : j(E_\ModuliParameter)=1\}$. This is certainly an upper bound for the number of geometric components of $\ModuliSpace''_0$ (in our concrete case, the degree turns out to be 630). We then simply find more than $630$ specialisations (over a finite field) along $\ModuliSpace''_0$ where $\sigma_\ModuliParameter$ takes distinct values, thus proving that the restriction of $\sigma_\ModuliParameter$ to $\ModuliSpace''_0$ is not constant on irreducible components.
\end{proof}

\section{Comments on the method}\label{sect: comments}
In this section we briefly discuss some features of our method and show how it departs from previously existing ones. To set the notation, start with a family of genus-2 curves $X_\alpha = \tilde{X}_\alpha \setminus \{q_\alpha\}$ (with a single point $q_\alpha$ at infinity) over some base $\ModuliSpace$, where $\alpha$ denotes a variable point in $\ModuliSpace$. Up to extending the base to some $\ModuliSpace' \to \ModuliSpace$, we can and will assume that all the data we consider (three torsion point, degree-3 cover, etcetera: see \Cref{setup}) are defined over $\ModuliSpace'$. Denote by $\beta$ a variable point in $\ModuliSpace'$, by $\tilde{Y}_\beta \to \tilde{X}_\beta$ the étale degree-3 cover constructed in our method, and by $p_{1, \beta}, p_{2, \beta}, p_{3, \beta}$ the points of $\tilde{Y}_\beta$ lying over the point at infinity of $\tilde{X}_\beta$. Suppose that, for a fixed $\beta$, the sections $p_{1, \beta}-p_{2,\beta}, p_{2,\beta}-p_{3,\beta}$ of $\Jac(\tilde{Y}_\beta)$ are torsion. Notice that the torsion orders are equal, since the two sections are related by an automorphism of $\operatorname{Jac}(\tilde{Y}_\beta)$, see the proof of \Cref{lemma: torsion on J_Y iff both torsion on E}. Denote by $n$ the exact common torsion order.

Analysing the proof of Bilu's criterion, we see that the desired effectivity is obtained as follows. Our torsion sections induce a morphism from $Y_\beta$ to an irreducible curve $Y^*$ defined by an equation $F(u_\beta,v_\beta)=0$ in $\G_m^2$. Here $Z$ is not a translate of a torus and $u_\beta, v_\beta$ are functions on $Y_\beta$ with divisors $n(p_{1, \beta}-p_{2, \beta})$, $n(p_{2, \beta}-p_{3, \beta})$.

These functions are multiplicatively independent modulo constants and assume $S$-unit values at the integral points of $Y_\beta$ which are lifted (using Chevalley-Weil) from those on $X_\beta$. The theory of Puiseux series and linear forms in logarithms gives the desired effectivity, as in the proof of Bilu's theorem.
We now describe in greater detail the image curve $Y^*$.

\begin{remark}[Degree of the image curve] Let us check that for large torsion orders 
the degree of the image curve $Y^* : F=0$ also increases. Since $n$ is the exact (common) torsion order, for every $m \geq 2$ the functions $u_\beta,v_\beta$ are not perfect $m$-th powers in $k(Y)$ (if $u_\beta = (u_\beta')^m$, the equation $m \operatorname{div}(u_\beta')=\operatorname{div}(u_\beta) = n(p_{1,\beta} - p_{2,\beta})$ shows that the torsion order is at most $n/m$; similarly for $v_\beta$). In the rest of the argument, we drop the subscript $\beta$ for convenience. 
 
We have an extension $k(Y)/k(u)$, totally ramified below $p_1$ and $p_2$ and of degree $n$ (indeed, once we know that the extension is totally ramified, the degree can be read off the divisor $\operatorname{div}(u) = n(p_1-p_2)$),  and we have an intermediate extension $k(u,v)/k(u)$. The extension $k(Y)/k(u,v)$, say of degree a divisor $m$ of $n$, is totally ramified below $p_1,p_2,p_3$ at least (by symmetry in the three points $p_1, p_2, p_3$).
   
Suppose first that the genus of $Y^*$, i.e.~of $k(u,v)$, is strictly greater than $1$. The map $Y \to Y^*$, of degree $m$, has at least $3$ ramification points of multiplicity at least $m-1$, so by the Hurwitz formula we have $6=2g(Y)-2\ge 2m +3(m-1)=5m-3$. We obtain $m=1$ and $k(u,v)=k(Y)$, so $F$ has degree at least (hence exactly) $n$. Indeed, the degree of $F$ is the degree of $k(Y)=k(u,v)$ over $k(u)$, which we have already proved to be $n$.

Suppose now that $g(Y^*)=1$. Then as above we obtain $6\ge 3m-3$, hence $m\le 3$ and we have a lower bound $n/3$ for the degree of $F$. 
   
Suppose finally that $g(Y^*)=0$. Then $k(u,v)=k(z)$ for some function $z\in k(Y)$. After a homography, we can assume that $u=z^q$ (since $k(Y)/k(u)$ is totally ramified above two points) and $v=c_1(z-c_2)^q$ for  constants $c_1,c_2\neq 0$ and $qm=n$. The divisor of $z$ would satisfy $q\cdot \operatorname{div}(z)=\operatorname{div}(u)=n(p_1-p_2)$, so $\operatorname{div}(z)=(n/q)(p_1-p_2)$. Since we are assuming that $n$ is the exact torsion order of $p_1-p_2$, we conclude that $q=1$, and therefore $c_3u+c_4v=1$ for some non-zero constants $c_3,c_4$. By the functional $abc$-theorem (Stothers-Mason theorem), since the total number of zeroes and poles of $u, v$ is three, we deduce $n\le 3+(2g(Y)-2)=9$, so $n$ is bounded, concluding the argument.
\end{remark}

\begin{remark}[Non-existence of universal equations for the method]
In our application of the method (say, for example, \Cref{T}), the integer $n$ grows to infinity, so we cannot perform the procedure sketched above in a uniform (bounded-degree) way for all curves in a given family. In fact, if $n$ is the exact torsion order of our divisors in $\operatorname{Jac}(\tilde{Y})$, the morphism $Y \to Y^*$ does not factor nontrivially through a morphism of lower degree to $\G_m^2$.

This is immaterial for the result itself, but it shows that one cannot prove effectivity by universal equations that apply to all members in the family (as would be possible, for instance, for the affine curves $y^2=f(x)$ with two points at infinity). The situation is somewhat reminiscent of Belyi's theorem.

Note that one can construct pencils where the relevant sections are identically torsion, see for example \S\ref{subsec: example with sections that are identically 2-torsion}. For such pencils, we have effectivity for all values of the parameter. However, this effectivity is in a sense less interesting, because it corresponds to a fixed type of diophantine equation (in the same way that all equations of type $y^2=f(x)$ may be effectively solved in integers).

\end{remark}  

\begin{remark}[Parameters $\beta$ for which there is effectivity]
It is possible to check effectively whether a given value of $\beta$ is or is not in our set: the sections $\phi(p_1)-\phi(p_2), \phi(p_2)-\phi(p_3)$ are explicit, and it is not hard to test effectively whether a point on an elliptic curve is torsion.
\end{remark}

\section{Explicit examples}\label{sect: examples}
In this final section, we collect some examples of families of genus-2 curves that exhibit different behaviours with respect to the objects considered in Setup \ref{setup}. The examples we give are rather arbitrary: the numerical coefficients are chosen to give comparatively simple equations, but don't have any intrinsic meaning. The reader could easily build many more examples simply by varying these numerical coefficients. Computer code to verify the claims made in this section is available at \cite{Computations}.

We will build most of our examples by first writing down a hyperelliptic equation of the form
\[
\tilde{X}_\ModuliParameter : y^2 = P_\ModuliParameter(x)^2-Q_\ModuliParameter(x)^3,
\]
where we choose the polynomials $P_\ModuliParameter(x), Q_\ModuliParameter(x)$ as a function of a parameter $\ModuliParameter$ in affine space (or, more precisely, in an open subset $U$ of affine space) and define $f_\ModuliParameter$ accordingly. We take as point at infinity $q_\ModuliParameter$ one of the two points with $x=0$; it will not matter which one of the two\footnote{there are two distinct points with $x=0$, provided that $P(0)^2-Q(0)^3 \neq 0$}, because the points $\phi(p_{1,\ModuliParameter}), \phi(p_{2,\ModuliParameter}), \phi(p_{3,\ModuliParameter})$ on
\[
E_\ModuliParameter : w^3 - 3Q_\ModuliParameter(x)w - 2P_\ModuliParameter(x)
\]
are those with $x$-coordinate equal to the $x$-coordinate of $q_\ModuliParameter$ (and in particular they do not depend on the $y$-coordinate of $q_\ModuliParameter$). Note that the section $q_\ModuliParameter$ is not defined over $U$, only over a double cover, but this is irrelevant for the construction of the examples. Finally, we write
\[
\sigma = (\sigma_1, \sigma_2) \colonequals (\phi(p_{1,\ModuliParameter})-\phi(p_{2,\ModuliParameter}), \phi(p_{2,\ModuliParameter})-\phi(p_{3,\ModuliParameter})).
\]

\subsection{Identically dependent components $\sigma_1, \sigma_2$}\label{subsect: identically dependent components}
We take
\[
P_\ModuliParameter(x) = b_3x^3, \quad Q_\ModuliParameter(x) = c_2x^2 + 3,
\]
where the parameter is $\ModuliParameter =(b_3, c_2)$. We have
\[
\phi(p_{1,\ModuliParameter}) = (0,3), \quad \phi(p_{2,\ModuliParameter})=(0,0), \quad \phi(p_{3,\ModuliParameter})=(0,-3).
\]
One checks that $(0,0)$ is a flex on $E_\ModuliParameter$, hence (since $\phi(p_{1,\ModuliParameter}), \phi(p_{2,\ModuliParameter}), \phi(p_{3,\ModuliParameter})$ lie on the same line) we have $\phi(p_{1,\ModuliParameter}) + \phi(p_{2,\ModuliParameter}) + \phi(p_{3,\ModuliParameter}) \sim 3\phi(p_{2,\ModuliParameter})$, and therefore the sections $\phi(p_{1,\ModuliParameter})-\phi(p_{2,\ModuliParameter})$ and $\phi(p_{2,\ModuliParameter})-\phi(p_{3,\ModuliParameter})$ coincide identically. In this case, the whole family $\tilde{X}_\ModuliParameter$ admits an automorphism of order $2$ (induced by $x \mapsto -x$) such that the corresponding quotient is a curve of genus $1$.

\subsection{The sections $\sigma_1, \sigma_2$ are generically independent}\label{subsec: sections generically independent}
For the next three examples we take
\[
P_\ModuliParameter(x) = b_3x^3 + b_1 x + 10, \quad Q_\ModuliParameter(x) = c_2 x^2 + 7,
\]
where the parameter is $\ModuliParameter=(b_3, b_1, c_2)$.
The choice of coefficients ensures that $P(0)^2 - Q(0)^3 \neq 0$ (so that the points with $x=0$ are non-special) and
\begin{equation}\label{eq: points at infinity on E}
 \phi(p_{1,\ModuliParameter}) = (0, -4), \quad \phi(p_{2,\ModuliParameter})=(0, -1), \quad \phi(p_{3,\ModuliParameter}) = (0, 5).
\end{equation}
\begin{remark}
    The coefficients $P(0)=10$ and $Q(0)=7$ are the minimal integers for which the points $\phi(p_{i,\ModuliParameter})$ have integer coordinates and are not flexes of $E_\ModuliParameter$.
\end{remark}

We now consider the point $\ModuliParameter_0$
with coordinates $(b_0, b_2, c_0) = (1, 0, -1)$. The sections $\sigma_1(\ModuliParameter_0), \sigma_2(s_0)$ provide two independent points of infinite order in $E_{\ModuliParameter_0}$: this can be tested by computing the height pairing between them and checking that it is non-degenerate. In other words, the group generated by $\sigma_{1, \ModuliParameter_0}, \sigma_{2, \ModuliParameter_0}$ in $E_{\ModuliParameter_0}(\mathbb{Q})$ is free of rank 2.
This implies that generically there are no linear relations between $\sigma_1$ and $\sigma_2$, and in particular, they are not identically torsion on $U$ (hence, a fortiori, there are no relations when we work over the whole moduli space of pointed genus-2 curves).

\subsection{A subvariety over which $\sigma_1$ is identically 2-torsion, but $\sigma_2$ is generically of infinite order}\label{subsect: sigma1 2-torsion, sigma2 infinite order}

Along the 2-dimensional subvariety
\[
b_3 = \frac{7}{27} b_1 c_2 - \frac{4}{3^9} b_1^3,
\]
the section $\sigma_1$ is identically of order 2 on $E_\ModuliParameter$, while $\sigma_{2}$ is generically of infinite order. To justify this last claim, one can take a further specialisation to the point $\ModuliParameter_1$ with coordinates $b_1=c_1=1$ (and $b_3$ given by the formula above) and check that the resulting point $\sigma_{2, \ModuliParameter_1}$ is not torsion. To find this example, we computed $2\sigma_1$ in parametric form (that is, as a section over the function field $\Q(b_1, b_3, c_2)$) and then imposed $2\sigma_1 = \infty$.
We also remark that $j$ is non-constant along this 2-dimensional family.

\subsection{Subvarieties giving torsion sections of higher order}
In principle, one can describe the surfaces along which one of the two sections $\sigma_i$ is killed by multiplication by any given $N$. However, these quickly become quite complicated: the (irreducible) surface along which $\sigma_1$ is of order 3 can again be computed by determining $3\sigma_1$ over the function field $\Q(b_1, b_3, c_2)$, and is given by
\[
\begin{aligned}
b1^8 & - \frac{837}{2}b_1^6c_2 + \frac{3645}{2}b_1^5b_3 + \frac{951345}{16}b_1^4c_2^2 - \frac{4113747}{8}b_1^3b_3c_2 + \frac{10451673}{16}b_1^2b_3^2 \\
& - \frac{42338133}{16}b_1^2c_2^3  + \frac{301327047}{8}b_1b_3c_2^2 -
        \frac{1420541793}{16}b_3^2c_2 - \frac{129140163}{4}c_2^4 = 0
 \end{aligned}
\]
A rational point of relatively small height on this surface is the point $\ModuliParameter_2$ with coordinates $(b_0, b_2, c_0) = (-9, 2, -11/4)$. At this specialisation, $\sigma_{1, \ModuliParameter_2}$ has order 3 and $\sigma_{2, \ModuliParameter_2}$ has infinite order.
\subsection{A family over which $\sigma_1, \sigma_2$ are both identically $2$-torsion}\label{subsec: example with sections that are identically 2-torsion}
Finally, the choice
\begin{equation}\label{eq: both identically torsion}
Q_\ModuliParameter(x) = x^2 + 7, \quad P_\ModuliParameter(x) = b_2 x^2 + 10, \quad \tilde{X}_\ModuliParameter : y^2 = f_\ModuliParameter(x) = P_\ModuliParameter(x)^2-Q_\ModuliParameter(x)^3
\end{equation}
(again with point at infinity having $x$-coordinate $0$) yields a pencil along which
the sections $\sigma_1, \sigma_2$ are both identically $2$-torsion. More precisely, with notation as in \eqref{eq: points at infinity on E} we have equalities of divisors
\[
2(\phi(p_{1,\ModuliParameter}))-2(\phi(p_{2,\ModuliParameter})) = \operatorname{div}\left( \frac{w+4}{w+1} \right), \quad 2(\phi(p_{2,\ModuliParameter}))-2(\phi(p_{3,\ModuliParameter})) = \operatorname{div}\left( \frac{w+1}{w-5} \right).
\] 
This example is found by considering a three-parameter family, working over the function field $\Q(b_2, b_3, c_2)$, and imposing $2\sigma_1 = 2\sigma_2=0$.
By computing invariants, one can also check that this family gives infinitely many geometrically non-isomorphic genus-2 curves. For this whole family of curves, Bilu's criterion (Theorem \ref{thm: Bilu version 3}) applies identically, yielding effectivity for the integral points. Note however that all the curves in the pencil possess the extra automorphism $x \mapsto -x$ which fixes the point at infinity.

\subsection{A concrete example}\label{subsec: concrete example, details}
We discuss the curve $X$ of \S\ref{subsubsec: concrete example}. Recall that $X$ is given by the planar model
\[
324x^4 - 324x^3 - 18x^2y + 71x^2 - 10xy - 6y^3 - 12y^2 = 0.
\]
As in Lemma \ref{lemma: singular model with one point at infinity}, this plane quartic has a unique singular point at $(0, 0)$ and a unique point $q$ at infinity, so it represents a curve of genus 2 with a marked point at infinity. 
The desingularisation of $X$ is isomorphic to the projective curve $\tilde{X}$ with hyperelliptic model
\[
y^2 = f(x) \colonequals 
360x^6 + 2052x^5 + 3969x^4 + 2916x^3 + 486x^2 + 729,
\]
with the point at infinity mapping to $(0, 27)$ in the hyperelliptic model (in particular, the point at infinity is not special). Using \cite{MR3882288, MR4280568} or \cite{MR1748293}, one can check that the Jacobian of $\tilde{X}$ is geometrically simple and does not have any non-trivial endomorphisms.
The decomposition
\[
f(x) = P(x)^2-Q(x)^3=\left(19x^3 + 54x^2 + 27x\right)^2-\left(x^2 -9\right)^3
\]
gives rise (see Lemma \ref{lemma: 3-to-1 parametrisation}) to a 3-torsion point on the Jacobian of $\tilde{X}$ and to an étale triple cover $\tilde{Y} \to \tilde{X}$ (see Setup \ref{setup}). There is an isogeny $\Jac(\tilde{Y}) \sim \Jac(\tilde{X}) \times E^2$, where $E$ is the elliptic curve given by \eqref{eq: affine equation for E}, namely
\[
E : w^3-3Q(x)w-2P(x) = w^3 - 3(x^2-9)w  -2\left(19x^3 + 54x^2 + 27x\right) = 0,
\]
where we can take the origin of the group law to be $(0,0)$.
By Proposition \ref{prop: properties of the construction} (6), the three points at infinity on $\tilde{Y}$ (that is, the three inverse images of $q$ under the map $\tilde{Y} \to \tilde{X}$) map to the three points on $E$ with $x=0$, that is,
\[
S_1 = (0,0), \quad S_2 = (0, 3\sqrt{3}i), \quad S_3 = (0, -3\sqrt{3}i).
\]
The differences $S_1-S_2, S_2-S_3$ are both $3$-torsion points (in fact, they are conjugate under the action of $\operatorname{Gal}(\overline{\Q}/\Q)$, hence -- since $E$ is defined over $\Q$ -- it suffices to check that one of them is torsion). By Remark \ref{rmk: how to apply Bilu's criterion} and Lemma \ref{lemma: torsion on J_Y iff both torsion on E}, this implies that there is effectivity for the integral points of $X = \tilde{X} \setminus \{q\}$.
This example is found by starting with a family similar to that of Section \ref{subsect: identically dependent components}, namely we take 
\[
P_\ModuliParameter(x) = b_3x^3+b_2x^2+b_1x, \quad Q_\ModuliParameter(x) = x^2 -9, \quad \text{ and } \quad q_\ModuliParameter=(0, 27),
\]
and then imposing that the points $S_2-S_1, S_2-S_3$ are torsion of a given order (in this case, $3$).

\subsection{Torsion of different orders}\label{subsect: different torsion orders}
Using a construction similar to that of the previous section, one shows that our method applies to the curve given by
\[
X : 81z^4 - 162z^3 + 9z^2w + 107z^2 - 62zw - 9w^3 + 36w^2=0,
\]
which is again a plane quartic with one singular point at the origin and one point at infinity. Here the two torsion sections have different orders: we have $2(S_1-S_2)=0$ and $3(S_2-S_3)=0$. The point at infinity is non-special; moreover, the Jacobian of $\tilde{X}$ is geometrically simple and does not have any non-trivial endomorphisms (we point out that in this case the criterion of \cite{MR1748293} does not suffice, and we need to rely on \cite{MR3882288}).

\subsection{A pencil with infinitely many effective examples}\label{subsect: pencil}
We describe a geometric method to obtain a 1-parameter family of genus-2 curves containing infinitely many for which there is effectivity for the integral points. Consider the Fermat cubic ${\mathcal F}:u^3+v^3=1$. In our pencil, the elliptic curves constructed by our method will all be isomorphic to ${\mathcal F}$. 

Working in the projective plane with homogeneous coordinates $(\LastCoord: U: V)$, where $u=U/\LastCoord$ and $v=V/\LastCoord$, we project $\mathcal{F}$ from a variable point $(0:\alpha:1)$ on the line at infinity to the line $V=0$. There are 6 branch points for this projection, given by the roots of the polynomial $f_\alpha(u) \colonequals (\alpha^3+1)u^6-2u^3+1$. We take our family of genus-2 curves to be
 \begin{equation*}
  \tilde{X}_\alpha : v^2=f_\alpha(u) = (\alpha^3+1)u^6-2u^3+1
 \end{equation*}
 over the $\alpha$-line.
 Since $(\alpha^3+1)u^6-2u^3+1=P^2-Q^3$ with $P=u^3-1$, $Q=-\alpha u^2$, we have a 3-torsion point defined over the base, and we can apply our method. One checks that the elliptic curve $E_\alpha$ is $Z^3+3\alpha ZU^2-2U^3+2\LastCoord^3=0$, isomorphic to $\mathcal{F}$ over some extension of $\Q(\alpha)$. Choosing the point at infinity on $\tilde{X}_\alpha$ in such a way that $\phi(p_2)=\phi(p_{2,\alpha})$ is a flex on $E_\alpha$ not on the line $\LastCoord=0$ gives the desired family of examples. We point out that the generic automorphism group of $\tilde{X}_\alpha$ is isomorphic to the dihedral group $D_6$ of order 12, but the point at infinity is neither special nor fixed by any automorphism (generically), hence we do not fall into a Thue or hyperelliptic case. This family contains infinitely many members for which our method gives effectivity for the integral points. The order of the torsion points to consider grows to infinity, so the comments of \Cref{sect: comments} apply: one cannot obtain the desired effectivity by `universal equations' along the entire family.

 \begin{remark}
     In this case, the base space of the family is an irreducible curve and the section $\sigma= (p_{i,1}-p_{i,2}, p_{i,2}-p_{i,3})$ is not identically torsion. 
The values of $\beta$ for which there is effectivity for the integral points of $X_\beta$ are algebraic numbers with the following properties: 
 \begin{enumerate}
     \item their heights are bounded, by well-known results of Silverman (or Demianenko-Manin, since we are in the isotrivial case); 
     \item their degrees tend to $\infty$ (this follows for example from the boundedness of height);
     \item they are topologically dense (for the complex topology).
 \end{enumerate}
 The same comments apply more generally to any family over an irreducible 1-dimensional base, provided that one can show that there are infinitely many torsion specialisations.
 \end{remark}

\bibliographystyle{alpha}
\bibliography{biblio.bib}

\end{document}